\documentclass[12pt,reqno]{amsart}
\usepackage{geometry}                
\usepackage{amsmath,amssymb}
\usepackage{graphicx}
\usepackage{amssymb,latexsym, amsmath}
\usepackage{amsfonts,amsthm}
\usepackage{amscd}
\usepackage{mathrsfs}
\usepackage{hyperref}
\usepackage{eucal}
\usepackage{epstopdf}
\usepackage{amsgen}
\usepackage{xspace}
\usepackage{verbatim}
\usepackage{stmaryrd}
\usepackage{enumitem}
\newlist{steps}{enumerate}{1}
\setlist[steps, 1]{label = Step \arabic*:}

\newcommand{\gd}{\Delta}

\newcommand{\tg}{\Tilde{g}}

\newcommand{\inpt}[1]{\langle #1 \rangle}

\newcommand{\gw}{\Omega}
\newcommand{\ap}{\alpha}

\newcommand{\gb}{\beta}

\newcommand{\ms}{\mathscr}

\newcommand{\nb}{\nabla}
\newcommand{\vp}{\varphi}
\newcommand{\ve}{\varepsilon}
\newcommand{\pdr}{\partial}

\newcommand{\tup}{\textup}
\newcommand{\csg}{\{ S(t)\}_{t\geq0}}

\newcommand{\beq}{\begin{equation}}
\newcommand{\eeq}{\end{equation}}
\newcommand{\bea}{\begin{align}}
\newcommand{\eea}{\end{align}}
\newcommand{\bthm}{\begin{theorem}}
\newcommand{\ethm}{\end{theorem}}
\newcommand{\bpr}{\begin{proof}}
\newcommand{\epr}{\end{proof}}
\newcommand{\bcl}{\begin{corollary}}
\newcommand{\ecl}{\end{corollary}}
\newcommand{\bpn}{\begin{proposition}}
\newcommand{\epn}{\end{proposition}}
\newcommand{\bre}{\begin{remark}}
\newcommand{\ere}{\end{remark}}
\newcommand{\bdf}{\begin{definition}}
\newcommand{\edf}{\end{definition}}
\newcommand{\bss}{\begin{align*}}
\newcommand{\ess}{\end{align*}}

\newcommand{\bl}{\label}

\newtheorem{theorem}{Theorem}[section]
\newtheorem{corollary}[theorem]{Corollary}

\newtheorem{lemma}[theorem]{Lemma}
\newtheorem{proposition}[theorem]{Proposition}

\theoremstyle{definition}
\newtheorem{definition}[theorem]{Definition}
\theoremstyle{remark}
\newtheorem{remark}{Remark}

\numberwithin{equation}{section}

\begin{document}

\title[Hindmarsh-Rose Equations]{Global Attractors for Hindmarsh-Rose Equations in Neurodynamics}

\author[C. Phan]{Chi Phan} 
\address{Department of Mathematics and Statistics, University of South Florida, Tampa, FL 33620, USA}
\email{chi@mail.usf.edu}

\author[Y. You]{Yuncheng You} 
\address{Department of Mathematics and Statistics, University of South Florida, Tampa, FL 33620, USA}
\email{you@mail.usf.edu}

\author[J. Su]{Jianzhong Su}
\address{Department of Mathematics, University of Texas at Arlington, Arlington, TX 76019, USA}
\email{Su@uta.edu}

\thanks{}




 \date{July 30, 2019}



\begin{abstract}
Global dynamics of the diffusive and partly diffusive Hindmarsh-Rose equations on a three-dimensional bounded domain originated in neurodynamics are investigated in this paper. The existence of global attractors as well as the regularity are proved through various uniform estimates showing the dissipative properties and the asymptotically compact characteristics, especially for the partly diffusive Hindmarsh-Rose equations by means of the Kolmogorov-Riesz theorem.
\end{abstract}

\maketitle

\section{\textbf{Introduction}}

The Hindmarsh-Rose equations for neuronal spiking-bursting of the intracellular membrane potential observed in experiments was originally proposed in \cite{HR1, HR2}. This mathematical model composed of three coupled nonlinear ordinary differential equations has been studied through numerical simulations and mathematical analysis in recent years, cf. \cite{HR1, HR2, IG, MFL, SPH, Su} and the references therein. It exhibits rich and interesting spatial-temporal bursting patterns, especially chaotic bursting and dynamics, 
as well as complex bifurcations. 

In this work we shall study the global dynamics in terms of the existence of a global attractor for the diffusive Hindmarsh-Rose equations, which is a new PDE model in neurodynamics:
\begin{align}
    \frac{\pdr u}{\pdr t} & = d_1 \gd u +  \vp (u) + v - w + J, \, \bl{ueq} \\
    \frac{\pdr v}{\pdr t} & = d_2 \gd v + \psi (u) - v, \bl{veq} \\
    \frac{\pdr w}{\pdr t} & = d_3 \gd w + q (u - c) - rw, \bl{weq}
\end{align}
for $t > 0,\; x \in \gw \subset \mathbb{R}^{n}$ ($n \leq 3$), where $\gw$ is a bounded domain with locally Lipschitz continuous boundary. The nonlinear terms 
\beq \bl{pp}
	\vp (u) = au^2 - bu^3, \quad \text{and} \quad \psi (u) = \alpha - \beta u^2.
\eeq 
The inject current $J $ is treated as a constant, but it can be a given function $J(x) \in L^2(\gw)$ and all the results in this paper remain valid. 

In this system \eqref{ueq}-\eqref{weq}, the variable $u(t,x)$ refers to the membrane electric potential of a neuronal cell, the variable $v(t, x)$ represents the transport rate of the ions of sodium and potassium through the fast ion channels and is called the spiking variable, while the variables $w(t, x)$ represents the transport rate across the neuronal cell membrane through slow channels of calcium and other ions correlated to the bursting phenomenon and is called the bursting variable. 

All the involved parameters are positive constants except $c \,(= u_R) \in \mathbb{R}$, which is a reference value of the membrane potential of a neuron cell. In the original model of ODE \cite{Su}, a set of the typical parameters are
\begin{gather*}
	J = 3.281, \;\; r = 0.0021, \;\; S = 4.0, \; \; q = rS,  \;\; c = -1.6,  \\[3pt]
	 \vp (s) = 3.0 s^2 - s^3, \;\; \psi (s) = 1.0 - 5.0 s^2.
\end{gather*}
We impose the Neumann boundary conditions for the three components,
\begin{equation} \label{nbc}
    \frac{\pdr u}{\pdr \nu} (t, x) = 0, \; \; \frac{\pdr v}{\pdr \nu} (t, x)= 0, \; \; \frac{\pdr w}{\pdr \nu} (t, x)= 0,\quad  t > 0,  \; x \in \partial \gw ,
\end{equation}
and the initial conditions 
\begin{equation} \bl{inc}
    u(0, x) = u_0 (x), \; v(0, x) = v_0 (x), \; w(0, x) = w_0 (x), \quad x \in \gw.
\end{equation}

We shall also consider the partly diffusive Hindmarsh-Rose equations
\begin{equation} \label{pHR}
	\begin{split}
    \frac{\pdr u}{\pdr t} & = d_1 \gd u +  \vp (u) + v - w + J,  \\
    \frac{\pdr v}{\pdr t} & =  \psi (u) - v,  \\
    \frac{\pdr w}{\pdr t} & =  q (u - c) - rw
    \end{split}
\end{equation}
and
\begin{equation} \label{qHR}
	\begin{split}
    \frac{\pdr u}{\pdr t} & = d_1 \gd u +  \vp (u) + v - w + J,  \\
    \frac{\pdr v}{\pdr t} & = d_2 \gd v + \psi (u) - v,  \\
    \frac{\pdr w}{\pdr t} & =  q (u - c) - rw.
    \end{split}
\end{equation}
In neuronal dynamics, the partly diffusive models \eqref{pHR} or \eqref{qHR} is more commonly interesting, since the ions currents may or may not diffuse. 

\subsection{\textbf{The Hindmarsh-Rose Model in ODE}}

In 1982-1984, J.L. Hindmarsh and R.M. Rose developed the mathematical model to describe neuronal dynamics:
\begin{equation} \label{HR}
	\begin{split}
    \frac{du}{dt} & = au^2 - bu^3 + v - w + J,  \\
    \frac{dv}{dt} & = \alpha - \beta u^2  - v,  \\
    \frac{dw}{dt} & =  q (u - u_R) - rw.
    \end{split}
\end{equation}

This neuron model was motivated by the discovery of neuronal cells in the pond snail \emph{Lymnaea} which generated a burst after being depolarized by a short current pulse. This model characterizes the phenomena of synaptic bursting and especially chaotic bursting in a three-dimensional $(u, v, w)$ space, which incorporates a third variable representing a slow current that hyperpolarizes the neuronal cell. 

Neuronal signals are short electrical pulses called spike or action potential. 
Neurons often exhibit bursts of alternating phases of rapid firing spikes and then quiescence. Bursting constitutes a mechanism to modulate and set the pace for brain functionalities and to communicate signals with the neighbor neurons. Bursting behaviors and patterns occur in a variety of bio-systems such as pituitary melanotropic gland, thalamic neurons, respiratory pacemaker neurons, and insulin-secreting pancreatic $\beta$-cells, cf. \cite{BRS, CK,CS, HR2}.

The mathematical analysis mainly using bifurcations together with numerical simulations of several models in ODEs on bursting behavior has been studied by many authors, cf. \cite{BB, DL, ET, MFL, Ri, SPH, Tr, WS, Su}. The more interesting study is on the behavior of neurons coupling and synchronization \cite{DFL, ET, Rv, SR}. 

Neurons communicate and coordinate actions through synapses or diffusive coupling called gap junction in neuroscience. Synaptic coupling of neurons has to reach certain threshold for release of quantal vesicles and synchronization \cite{DJ, Ru, SC}. 

The chaotic coupling exhibited in the simulations and analysis of this Hindmarsh-Rose model in ordinary differential equations shows more rapid synchronization and more effective regularization of neurons due to \emph{lower threshold} than the synaptic coupling \cite{Tr, Su}. It was rigorously proved in \cite{SPH, Su} that chaotic bursting solutions can be quickly synchronized and regularized when the coupling strength is large enough to topologically change the bifurcation diagram based on this Hindmarsh-Rose model, but the dynamics of chaotic bursting is highly complex.

It is known that Hodgkin-Huxley equations \cite{HH} (1952) provided a four-dimensional model for the dynamics of membrane potential taking into account of the sodium, potassium as well as leak ions current. It is a highly nonlinear system if without simplification assumptions. FitzHugh-Nagumo equations \cite{FH} (1961-1962) derived a two-dimensional model for an excitable neuron with the membrane potential and the current variable. This model admits an exquisite phase plane analysis showing spikes excited by supra-threshold input pulses and sustained periodic spiking with refractory period, but due to the 2D nature FitzHugh-Nagumo equations exclude any chaotic solutions and chaotic dynamics so that no chaotic bursting can be generated. 


It has been indicated by research 
that the Hindmarsh-Rose model in ODE causes lower the neuron firing threshold. 
More importantly, this Hindmarsh-Rose model allows varying interspike-interval. Therefore, this 3D model is a suitable choice for the investigation of both the regular bursting and the chaotic bursting when the parameters vary. The study of dynamical properties of the Hindmarsh-Rose equations \eqref{HR} as a neuron model exposes to a wide range of applications in neuroscience. 
 
The rest of Section 1 is the formulation of the system \eqref{ueq}-\eqref{weq} and provides basic concepts and results in the theory of global dynamics. In Section 2 we shall conduct uniform estimates to show the absorbing properties of the Hindmarsh-Rose semiflow in $L^{2p}$ spaces for $1 \leq p \leq 3$. In Section 3, the main result on the existence of global attractor for the diffusive Hindmarsh-Rose system is proved. Section 4 will show the regularity and structure of the global attractor. Finally, in Section 5 we shall prove the asymptotic compactness of the two partly diffusive systems \eqref{pHR} and \eqref{qHR} by means of the Kolmogorov-Riesz theorem and the existence of global attractors. 

\subsection{\textbf{Formulation and Preliminaries}}

Neuron is a specialized biological cell in the brain and the central nervous system. In general, neurons have four parts: the central cell body containing the nucleus and intracellular organelles, the dendrites, the axon, and the terminals. The dendrites are the short branches near the nucleus receiving incoming signals of voltage pulse. The axon is a long branch to propagate outgoing signals. The nerve terminals communicate these pulse signals to other neurons. 

Neurons are immersed in aqueous chemical solutions consisting of different diffusive ions electrically charged. 
The voltage and concentration-dependent conductances of these various ions can give rise to different neural behavior. As pointed out in \cite{EI}, neuron is a distributed dynamical system. 

From physical and mathematical point of view, it is meaningful and useful to consider the Hindmarsh-Rose model in partial differential equations with the spatial variables $x$ involved. 
Here in the abstract extent, we shall study the diffusive Hindmarsh-Rose equations \eqref{ueq}-\eqref{weq} and the partly diffusive models \eqref{pHR} and \eqref{qHR} in a bounded domain of space $\mathbb{R}^3$ and focus on the global dynamics of the solutions.

We start with formulation of the aforementioned initial-boundary value problem of \eqref{ueq}--\eqref{inc} into an abstract evolutionary equation. Define the Hilbert space $H = [L^2 (\gw)]^3 = L^2 (\gw, \mathbb{R}^3)$ and the Sobolev space $E =  [H^{1}(\gw)]^3 = H^1 (\gw, \mathbb{R}^3)$. The norm and inner-product of $H$ or $L^2 (\gw)$ will be denoted by $\| \, \cdot \, \|$ and $\inpt{\,\cdot , \cdot\,}$, respectively. The norm of $E$ will be denoted by $\| \, \cdot \, \|_E$. The norm of $L^p (\gw)$ or $L^p (\gw, \mathbb{R}^3)$ will be dented by $\| \cdot \|_{L^p}$ if $p \neq 2$. We use $| \, \cdot \, |$ to denote a vector norm in a Euclidean space.

The initial-boundary value problem \eqref{ueq}--\eqref{inc} is formulated as an initial value problem of the evolutionary equation:
\begin{equation} \label{pb}
 	\begin{split}
   	& \frac{\partial g}{\partial t} = A g + f(g), \quad t > 0, \\
    	g &\, (0) = g_0 = (u_0, v_0, w_0) \in H.
	\end{split}
\end{equation}
Here the nonnegative self-adjoint operator
\begin{equation} \label{opA}
        A =
        \begin{pmatrix}
            d_1 \gd  & 0   & 0 \\[3pt]
            0 & d_2 \gd  & 0 \\[3pt]
            0 & 0 & d_3 \gd
        \end{pmatrix}
        : D(A) \rightarrow H,
\end{equation}
where $D(A) = \{g \in H^2(\gw, \mathbb{R}^3): \pdr g /\pdr \nu = 0 \}$ is the generator of an analytic $C_0$-semigroup $\{e^{At}\}_{t \geq 0}$ on the Hilbert space $H$ due to the Lumer-Phillips theorem  \cite{SY}. By the fact that $H^{1}(\gw) \hookrightarrow L^6(\gw)$ is a continuous imbedding for space dimension $n \leq 3$ and by the H\"{o}lder inequality, there is a constant $C_0 > 0$ such that 
$$
    \| \vp (u)  \| \leq C_0 \| u \|_{L^6}^3 \quad \tup{and} \quad \|\psi (u) \| \leq C_0 \| u \|_{L^4}^2 \quad \textup{for} \; u \in L^6 (\gw).
$$
Therefore, the nonlinear mapping 
\begin{equation} \label{opf}
    f(u,v, w) =
        \begin{pmatrix}
             \vp (u) + v - w + J \\[4pt]
            \psi (u) - v,  \\[4pt]
	     q (u - c) - rw
        \end{pmatrix}
        : E \longrightarrow H
\end{equation}
is a locally Lipschitz continuous mapping. We can simply write column vectors $g(t)$ as $(u(t, \cdot), v(t, \cdot ), w(t, \cdot))$ and write $g_0 = (u_0, v_0, w_0)$. Consider the weak solution of this initial value problem \eqref{pb}, defined below \cite{CV}: 
\begin{definition} \label{D:wksn}
	A function $g(t, x), (t, x) \in [0, \tau] \times \gw$, is called a \emph{weak solution} to the initial value problem \eqref{pb}, if the following conditions are satisfied:
	
	\textup{(i)} $\frac{d}{dt} (g, \zeta) = (Ag, \zeta) + (f(g), \zeta)$ is satisfied for a.e. $t \in [0, \tau]$ and for any $\zeta \in E$;
	
	\textup{(ii)} $g(t, \cdot) \in L^2 (0, \tau; E) \cap C_w ([0, \tau]; H)$ such that $g(0) = g_0$.
	
\noindent
Here $(\cdot , \cdot)$ stands for the dual product of $E^*$ and $E$, and $C_w$ stands for the weakly continuous functions valued in $H$. \end{definition}

\begin{lemma} \label{Lwn}
	For any given initial data $g_0 \in H$, there exists a unique local weak solution $g(t, g_0) = (u(t), v(t), w(t)), \, t \in [0, \tau]$, for some $\tau > 0$, of the initial value problem \eqref{pb}, which satisfies
\begin{equation} \label{soln}
    	g \in C([0, T_{max}); H) \cap C^1 ((0, T_{max}); H) \cap L_{loc}^2 ([0, T_{max}); E),
\end{equation}
where $I_{max} = [0, T_{max})$ is the maximal interval of existence. Moreover, any weak solution $g(t, g_0)$ becomes a strong solution for $t > 0$ and has the regularity
\beq \bl{stsl}
	g \in C([t_0, T_{max}); E) \cap C^1 ((t_0, T_{max}); H) \cap L_{loc}^2 ([t_0, T_{max}); H^2 (\gw, \mathbb{R}^3))
\eeq
for any\, $t_0 \in (0, T_{max})$.
\end{lemma}
\begin{proof}
The proof of the existence and uniqueness of a weak solution is made by conducting \emph{a priori} estimates on the Galerkin approximate solutions of the initial value problem \eqref{pb}, these estimates are similar to what we shall present in Section 2, and then by the Lions-Magenes type of weak and weak$^*$ compactness and convergence argument \cite{CV, SY}. The statement of strong solution follows from the parabolic regularity of the evolutionary equations \cite{SY, Tm}. The details are omitted here.
\end{proof}

We do not assume the initial data $u_0, v_0, w_0$, nor the solutions $u(t, x)$, $v(t, x), w(t, x)$, are nonnegative functions. We do not impose any conditions on any of the positive parameters, nor on the parameter $c \in \mathbb{R}$ of the equation \eqref{weq}, in the proof of all the results in this paper.

Here the goal is to prove the existence and regularity of a global attractor, which will characterize qualitatively the longtime, asymptotic, and global dynamics of all the solution trajectories of this PDE system \eqref{pb} and the hybrid PDE-ODE systems \eqref{pHR} and \eqref{qHR}. As specified in \cite{CV, Rb, SY, Tm} as well as in \cite{Y08, Y10}, global attractor is a depository (usually fractal finite-dimensional) of all the permanent regimes including all steady states, periodic or knotted or chaotic orbits, and unstable manifolds for an infinite-dimensional dynamical system. These topics of global dynamic patterns are also important in neural field and neural network theories \cite{Co, ET}. For the autocatalytic and Boissonade reaction-diffusion systems \cite{Tu, Y08, Y10, Y12}, it is proved that global attractors exist. 

We refer to \cite{CV, Rb, SY, Tm} for the concepts and basic facts in the theory of infinite dimensional dynamical systems, including the few listed here for clarity.

\begin{definition} \label{Dabsb}
Let $\{S(t)\}_{t \geq 0}$ be a semiflow on a Banach space $\ms{X}$. A bounded set $B_0$ of $\ms{X}$ is called an absorbing set for this semiflow, if for any given bounded subset $B \subset \ms{X}$ there is a finite time $T_0 \geq 0$ depending on $B$, such that $S(t)B \subset B_0$ for all $t \geq T_0$.
\end{definition}

\begin{definition} \label{Dasmp}
A semiflow $\{S(t)\}_{t \geq 0}$ on a Banach space $\ms{X}$ is called asymptotically compact if for any bounded sequence $\{w_n \}$ in $\ms{X}$ and any monotone increasing sequences $0 < t_n \to \infty$, there exist subsequences $\{w_{n_k}\}$ of $\{w_n \}$ and $\{t_{n_k}\}$ of $\{t_n\}$ such that $\lim_{k \to \infty}
S(t_{n_k}) w_{n_k}$ exists in $\ms{X}$.
\end{definition}

\begin{definition}[Global Attractor] \label{Dgla}
A set $\mathscr{A}$ in a Banach space $\ms{X}$ is called a global attractor for a semiflow $\csg$ on $\ms{X}$, if the following two properties are satisfied:

(i) $\mathscr{A}$ is a nonempty, compact, and invariant set in the space $\ms{X}$. 

(ii) $\mathscr{A}$ attracts any given bounded set $B \subset \ms{X}$ in the sense 
$$
	\text{dist}_{\ms{X}} (S(t)B, \mathscr{A}) = \sup_{x \in B} \inf_{y \in \mathscr{A}} \| S(t)x - y \|_{\ms{X}} \to 0, \;\;  \text{as} \; \; t \to \infty.
$$
\end{definition}

The following is the main existing result on the existence of a global attractor. 
\begin{proposition} \cite{CV, Rb, SY, Tm} \label{L:basic}
Let $\{S(t)\}_{t\geq 0}$ be a semiflow on a Banach space $\ms{X}$. If the following two conditions are satisfied\textup{:}

\textup{(i)} there exists a bounded absorbing set $B_0 \subset \ms{X}$ for $\{S(t)\}_{t\geq 0}$, and

\textup{(ii)} the semiflow $\{S(t)\}_{t\geq 0}$ is asymptotically compact on $\ms{X}$,

\noindent
then there exists a global attractor $\ms{A}$ in $\ms{X}$ for the semiflow $\{S(t)\}_{t\geq 0}$ and the global attractor is given by
\beq \bl{glafm}
        \ms{A} = \bigcap_{\tau \geq 0} \; \overline{\bigcup_{t \geq \tau} (S(t)B_0)}.
\eeq
\end{proposition}

\begin{definition} \label{Dxya}
Let $\{S(t)\}_{t \geq 0}$ be a semiflow on a Banach space $X$ and let $Y$ be a Banach space which is compactly embedded in $X$. Then a set $\mathscr{A} \subset Y$ is called an $(X, Y)$-global attractor for this semiflow if the following two conditions are satisfied:

(i) $\mathscr{A}$ is a  nonempty, compact, and invariant set in $Y$, and

(ii) $\mathscr{A}$ attracts any bounded set $B$ of $X$ with respect to the $Y$-norm.
\end{definition}

The Gagliardo-Nirenberg inequalities \cite[Appendtx B]{SY} of interpolation is useful in estimates of solutions of partial differential equations:
\beq \label{GN}
	\| y \|_{W^{k, p} (\gw)} \leq C \| y \|^\theta_{W^{m, q} (\gw)}\,  \| y \|^{1 - \theta}_{L^r (\gw)}, \quad \text {for all} \;\; y \in W^{m, q} (\gw),
\eeq
where $C > 0$ is a constant, provided that $p, q, r \geq 1, 0 < \theta < 1$, and
$$
	k - \frac{n}{p}\,  \leq \, \theta \left(m - \frac{n}{q}\right)  - (1 - \theta)\, \frac{n}{r},  \quad n = \text{dim}\, (\gw).
$$
The Young's inequality in the general form for any nonnegative $x, y$ is
\beq \bl{Hld}
	xy  \leq \ve x^p + C(\ve, p) y^q, \qquad \frac{1}{p} + \frac{1}{q} = 1, \, (p, q \geq 1),  \qquad C(\ve, p) = \ve^{-q/p}.
\eeq
with $\ve > 0$ which can be arbitrarily small.

In the sequel, we often write $u(t, x), v(t, x), w(t, x)$ as $u(t), v(t), w(t)$ or even as $u, v, w$ for brevity. We shall use $C$ to denote a generic constant whose value depends on the context. Otherwise, it will be marked as $C(\ve, D)$ if $C$ depends on a constant $\ve$ and a given set $D$ or maybe on more quantities.

\section{\textbf{Uniform Estimates and Absorbing Properties}}

In this section we shall conduct scaled \emph{a priori} estimates to show that the weak solution of the problem \eqref{pb} exists globally in time and the solution semiflow is dissipative meaning there exists an absorbing set in $H$. We shall also prove the absorbing properties  in the spaces $L^4 (\gw, \mathbb{R}^3)$ and $L^6 (\gw, \mathbb{R}^3)$ for each trajectory, which will play a key role in Section 5 for the study of  the partly diffusive Hindmarsh-Rose equations \eqref{pHR} and \eqref{qHR}. 

\subsection{\textbf{Global Existence and Dissipative Property in $H$}}

\begin{theorem} \label{Lm2}
For any given initial data $g_0 = (u_0, v_0, w_0) \in H$, there exists a unique global weak solution in time, $g(t) = (u(t), v(t), w(t)), \, t \in [0, \infty)$, of the initial value problem \eqref{pb} for the diffusive Hindmarsh-Rose equations \eqref{ueq}-\eqref{weq}. The weak solution becomes a strong solution on the interval $(0, \infty)$. 
\end{theorem}

\begin{proof}
Taking the $L^2$ inner-product $\inpt{\eqref{ueq}, C_1 u(t)}$ with an adjustable constant $C_1 > 0$ to be determined later, we use the Young's inequality to get
\begin{equation} \label{u1}
	\begin{split}
	\frac{C_1}{2} \frac{d}{dt} \|u \|^2 + C_1 d_1 \| \nabla u\|^2 = &\, \int_\gw C_1 (\vp (u) u + uv - uw +Ju)\, dx \\
	= &\int_\gw C_1 (au^3 -bu^4 + uv - uw +Ju)\, dx.  
	\end{split}
\end{equation}
Taking the $L^2$ inner-products $\inpt{\eqref{veq}, v(t)}$ and  $\inpt{\eqref{weq}, w(t)}$, we have
\begin{equation} \label{v1}
	\begin{split}
	&\frac{1}{2} \frac{d}{dt} \|v \|^2 + d_2 \| \nabla v\|^2 = \int_\gw  (\psi (u) v - v^2)\, dx = \int_\gw (\ap v - \gb u^2 v - v^2)\, dx \\[2pt]
	\leq &\int_\gw \left(\ap v +\frac{1}{2} (\gb^2 u^4 + v^2) - v^2\right) dx = \int_\gw \left(\ap v +\frac{1}{2} \gb^2 u^4 - \frac{1}{2} v^2\right) dx \\
	\leq & \int_\gw \left(2\ap^2 + \frac{1}{8} v^2 +\frac{1}{2} \gb^2 u^4 - \frac{1}{2} v^2\right) dx = \int_\gw \left(2\ap^2 +\frac{1}{2} \gb^2 u^4 - \frac{3}{8} v^2\right) dx
	\end{split}
\end{equation}
and
\begin{equation} \label{w1}
	\begin{split}
	&\frac{1}{2} \frac{d}{dt} \|w \|^2 + d_3 \| \nabla w\|^2 = \int_\gw (q (u - c)w - rw^2)\, dx  \\
	\leq & \int_\gw \left(\frac{q^2}{2r} (u - c)^2 + \frac{1}{2} r w^2 - r w^2 \right) dx \leq \int_\gw \left(\frac{q^2}{r} (u^2 + c^2) - \frac{1}{2} r w^2 \right) dx.
	\end{split}
\end{equation}
Now we choose the positive constant in \eqref{u1} to be $C_1 = \frac{1}{b} (\gb^2 + 4)$, so that
$$
	\int_\gw (- C_1 b u^4)\, dx + \int_\gw (\gb^2 u^4)\, dx \leq \int_\gw (-4 u^4)\, dx.
$$ 
Then we estimate all the mixed product terms on the right-hand side of the above three inequalities by using the Young's inequality in an appropriate way as follows. In \eqref{u1},
\begin{gather*}
	\int_\gw C_1 au^3\, dx \leq \frac{3}{4} \int_\gw u^4\, dx + \frac{1}{4}\int_\gw (C_1 a)^4 \, dx \leq \int_\gw u^4\, dx + (C_1 a)^4 |\gw|, \\
	\int_\gw C_1 (uv - uw + Ju)\, dx \leq \int_\gw \left(2(C_1 u)^2 + \frac{1}{8} v^2 + \frac{(C_1 u)^2}{r} + \frac{1}{4} r w^2 + C_1 u^2 + C_1J^2 \right) dx,
\end{gather*}
where on the right-hand side of the second inequality we further treat the three terms involving $u^2$ as
$$
	\int_\gw \left(2(C_1 u)^2 + \frac{(C_1 u)^2}{r} + C_1 u^2 \right) dx \leq \int_\gw u^4 \, dx +\left[C_1^2 \left(2 +\frac{1}{r}\right) + C_1\right]^2 |\gw |.
$$
Then in \eqref{w1},
\begin{gather*}
	\int_\gw \frac{1}{r} q^2 u^2 \, dx \leq \int_\gw \left(\frac{u^4}{2} + \frac{q^4}{2r^2}\right) dx  \leq \int_\gw u^4\, dx + \frac{q^4}{r^2} |\gw|.
\end{gather*}
Substitute the above term estimates into \eqref{u1} and \eqref{w1} and then sum up the three inequalities \eqref{u1}-\eqref{w1} to obtain
\beq \label{g2}
	\begin{split}
	&\frac{1}{2} \frac{d}{dt} (C_1 \|u\|^2 +  \|v\|^2 +  \|w\|^2) + (C_1 d_1 \|\nb u\|^2 + d_2 \|\nb v\|^2 +d_3\|\nb w\|^2) \\
	\leq & \int_\gw C_1 (au^3 -bu^4 + uv - uw +Ju)\, dx \\
	+ &  \int_\gw \left(2\ap^2 +\frac{1}{2} \gb^2 u^4 - \frac{3}{8} v^2\right) dx + \int_\gw \left(\frac{q^2}{r} (u^2 + c^2) - \frac{1}{2} r w^2 \right) dx \\
	\leq & \int_\gw (3 - 4)u^4\, dx + \int_\gw \left(\frac{1}{8} - \frac{3}{8}\right) v^2\, dx + \int_\gw \left(\frac{1}{4} - \frac{1}{2} \right) rw^2\, dx \\
	+ & \, |\gw | \left( (C_1 a)^4 + C_1 J^2  + \left[C_1^2 \left(2 +\frac{1}{r}\right) + C_1\right]^2 + 2\ap^2 + \frac{q^2 c^2}{r} + \frac{q^4}{r^2} \right) \\
	= &\, - \int_\gw \left(u^4(t, x) + \frac{1}{4} v^2 (t, x) + \frac{1}{4} rw^2 (t, x)\right) dx + C_2 |\gw | 
	\end{split}
\eeq
where $C_2 > 0$ is the constant given by 
$$
	C_2 = (C_1 a)^4 + C_1 J^2  + \left[C_1^2 \left(2 +\frac{1}{r}\right) + C_1\right]^2 + 2\ap^2 + \frac{q^2 c^2}{r} + \frac{q^4}{r^2}.
$$
Set 
$$
	 d = 2\min \{d_1, d_2, d_3\}.
$$
Then \eqref{g2} yields the following uniform group estimate,
\beq \label{E1}
	\begin{split}
	&\frac{d}{dt} (C_1 \|u(t)\|^2 + \|v(t)\|^2 + \|w(t)\|^2)  + d (C_1 \|\nb u\|^2 + \|\nb v\|^2 + \| \nb w\|^2) \\
	+ &\, \int_\gw \left(2u^4(t, x) + \frac{1}{2} v^2 (t, x) + \frac{1}{2} rw^2 (t, x)\right) dx  \leq 2C_2 |\gw|, 
	\end{split}
\eeq
for $t \in I_{max} = [0, T_{max})$, the maximal time interval of solution existence. Since 
$$ 
	2u^4 \geq \frac{1}{2} \left(C_1 u^2 - \frac{C_1^2}{16}\right), 
$$
it follows from \eqref{E1} that
\begin{equation*}
	\begin{split}
	&\frac{d}{dt} (C_1 \|u(t)\|^2 + \|v(t)\|^2 + \|w(t)\|^2)  + d (C_1 \|\nb u\|^2 + \|\nb v\|^2 + \|\nb w\|^2) \\
	+ &\, \int_\gw \frac{1}{2} \left(C_1 u^2(t, x) + v^2 (t, x) + r w^2 (t, x)\right) dx \leq \left(2C_2 + \frac{C_1^2}{32}\right) |\gw |.
	\end{split}
\end{equation*}
Set $r_1 = \frac{1}{2} \min \{1, r\}$. Then we have 
\begin{equation} \label{E2}
	\begin{split}
	\frac{d}{dt} (C_1 \|u(t)\|^2 &+ \|v(t)\|^2 + \|w(t)\|^2)  + d (C_1 \|\nb u\|^2 + \|\nb v\|^2 + \| \nb w\|^2) \\
	+ &\, r_1 (C_1 \| u\|^2 + \| v \|^2 + \|w\|^2) \leq \left(2C_2 + \frac{C_1^2}{32}\right) |\gw |.
	\end{split}
\end{equation}
Apply the Gronwall inequality to 
$$
	\frac{d}{dt} (C_1 \|u(t)\|^2 + \|v(t)\|^2 + \|w(t)\|^2)  +  r_1 (C_1 \| u\|^2 + \| v\|^2 + \|w\|^2) \leq \left(2C_2 + \frac{C_1^2}{32}\right) |\gw |
$$
and we obtain
\beq \label{dse}
	C_1 \|u(t)\|^2 + \|v(t)\|^2 + \|w(t)\|^2 \leq e^{- r_1 t} (C_1 \|u_0\|^2 + \|v_0\|^2 + \|w_0\|^2) + M |\gw |
\eeq 
where 
$$
	M = \frac{1}{r_1}\left(2C_2 + \frac{C_1^2}{32}\right) .
$$
The estimate \eqref{dse} shows that the weak solution will never blow up at any finite time because it is uniformly bounded,
$$
	C_1 \|u(t)\|^2 + \|v(t)\|^2 + \|w(t)\|^2 \leq C_1 \|u_0\|^2 + \|v_0\|^2 + \|w_0|^2 + M |\gw |.
$$
Therefore the weak solution of the initial value problem \eqref{pb} exists globally in time for any initial data. The time interval of maximal existence is always $[0, \infty)$.
\end{proof}

The global existence and uniqueness of the weak solutions and their continuous dependence on the initial data enable us to define the solution semiflow of the diffusive Hindmarsh-Rose equations \eqref{ueq}-\eqref{weq} on the space $H$ as follows:
$$
	S(t): g_0 \longmapsto g(t, g_0) = (u(t, \cdot), v(t, \cdot), w(t, \cdot)), \;\; g_0 = (u_0, v_0, w_0) \in H, \;\; t \geq 0,
$$
where $g(t, g_0)$ is the weak solution with the initial status $g(0) = g_0$. We shall call this semiflow $\{S(t)\}_{t \geq 0}$ the Hindmarsh-Rose semiflow associated with the formulated evolutionary equation \eqref{pb}. 

\begin{theorem} \label{Th2}
	There exists an absorbing set in the space $H$ for the Hindmarsh-Rose semiflow $\{S(t)\}_{t \geq 0}$, which is the bounded ball 
\beq \label{abs}
	B_0 = \{ g \in H: \| g \|^2 \leq K_1\}
\eeq 
where $K_1 = \frac{M |\gw |}{\min \{C_1, 1\}} + 1$.
\end{theorem}

\begin{proof}
From the uniform estimate \eqref{dse} in Theorem \ref{Lm2} we see that 
\beq \label{lsp2}
	\limsup_{t \to \infty} \, (\|u(t)\|^2 + \|v(t)\|^2 + \|w(t)\|^2 ) < K_1 = \frac{M |\gw |}{\min \{C_1, 1\}} + 1
\eeq
for all weak solutions of \eqref{pb} with any initial data $g_0 \in H$. Moreover, for any given bounded set $B = \{g \in H: \|g \| \leq R\}$ in $H$, there exists a finite time 
\beq \label{T0B}
	T_0 (B) = \frac{1}{r_1} \log \, (R^2 \max \{C_1, 1\})
\eeq
such that $\|u(t)\|^2 + \|v(t)\|^2 + \|w(t)\|^2 < K_1$ for all $t > T_0 (B)$ and for all $g_0 \in B$. Thus, by Definition \eqref{Dabsb}, the bounded ball $B_0$ is an absorbing set and the Hindmarsh-Rose semiflow is dissipative in the phase space $H$. 
\end{proof}

\begin{corollary} \label{cor1}
	There also exists an absorbing set in the space $H$ for the semiflow generated by the partly diffusive Hindmarsh-Rose equations \eqref{pHR} and \eqref{qHR}.
\end{corollary}

\begin{proof}
	The proof of Lemma \eqref{Lm2} is still valid without the terms of $\|\nb v \|^2$ and $\| \nb w \|^2$. Instead of \eqref{E2} we have the inequality 
\begin{equation} \label{pE2}
	\begin{split}
	\frac{d}{dt} &(C_1 \|u(t)\|^2 + \|v(t)\|^2 + \|w(t)\|^2)  + d _1C_1 \|\nb u\|^2  \\
	+ &\, r_1 (C_1 \| u\|^2 + \| v \|^2 + \|w\|^2) \leq \left(2C_2 + \frac{C_1^2}{32}\right) |\gw |
	\end{split}
\end{equation}
It leads to the same result \eqref{dse}. Therefore, the same argument as in the proof of Theorem \eqref{Th2} shows that there exists an absorbing set for the semiflow generated by \eqref{pHR} in the phase space $H$. Similarly the result holds for the semiflow of \eqref{qHR}.
\end{proof}

\subsection{\textbf{Absorbing Property of the Hindmarsh-Rose Semiflow in $L^4$ and $L^6$}}

The following theorem claims the absorbing property in higher-order integrable spaces for each solution trajectory. It plays a key role to establish the asymptotic compactness of the solution semiflow associated with the partly diffusive Hindmarsh-Rose equations \eqref{pHR} and \eqref{qHR} in Section 5  as well as the $H^2$-regularity of the global attractor in Section 4. 
\begin{theorem} \label{Th2p}
	For $p = 2$ and $3$ respectively, there exists a constant $K_p > 0$ such that the absorbing property with respect to the space $L^{2p} (\gw, \mathbb{R}^3)$,
\beq \label{Kp}
	\limsup_{t \to \infty} \, \left(\|u(t)\|^{2 p}_{L^{2p}} + \|v(t)\|^{2 p}_{L^{2p}} + \|w(t)\|^{2 p}_{L^{2p}} \right) < K_p
\eeq
is satisfied by every weak solution $S(t) g_0 = g(t, g_0)) = (u(t), v(t), w(t))$ of the diffusive Hindmarsh-Rose equations \eqref{pb} for any geven initial state $g_0 \in H$. 
\end{theorem} 

\begin{proof}
	Based on Lemma \ref{Lm2} and Theorem \ref{Th2}, the weak solution $(u(t), v(t), w(t))$ of the Hindmarsh-Rose evolutionary equation \eqref{pb} exists globally in time $t \in [0, \infty)$ for any initial state $g_0 = (u_0, v_0, w_0) \in H$. By Lemma \ref{Lwn}, there exists a time $t_0 \in (0, 1)$ such that for space dimension $n \leq 3$,
$$
	S(t_0) g_0 = g(t_0, g_0) \in E = H^1 (\gw, \mathbb{R}^3) \subset L^6 (\gw, \mathbb{R}^3) \subset L^4 (\gw, \mathbb{R}^3). 
$$
According to Lemma \ref{Lwn} the weak solution $S(t)g_0$ becomes a strong solution on $[t_0, \infty)$ and satisfies
$$
	S(\cdot)g_0 \in C([t_0, \infty), E) \subset C([t_0, \infty), L^6 (\gw, \mathbb{R}^3)) \subset C([t_0, \infty), L^4 (\gw, \mathbb{R}^3)).
$$
This parabolic regularity with time shifting enables us to simply assume that the initial state $g_0 \in E \subset L^6 (\gw, \mathbb{R}^3) \subset L^4 (\gw, \mathbb{R}^3))$ and, by the bootstrap argument for resulting strong solutions, even assume that $g_0 \in H^2 (\gw, \mathbb{R}^3) \subset L^8 (\gw, \mathbb{R}^3)$ in proving the longtime dynamical property \eqref{Kp}.

Step 1. Common estimates.

For $p = 2$ and $p= 3$, we take the $L^2$ inner-product  $\inpt{\eqref{ueq}, u^{2p-1}}$ to get
\beq \label{u2}
	\begin{split}
	&\frac{1}{2p} \frac{d}{dt} \|u(t)\|^{2p}_{L^{2p}} + d_1 (2p-1) \|u^{p-1} \nb u \|^2_{L^2} \\[3pt]
	= &\,  \frac{1}{2p} \frac{d}{dt} \|u(t)\|^{2p}_{L^{2p}} + \frac{1}{p^2} d_1 (2p - 1) \| \nb (u^p)\|^2_{L^2} \\[3pt]
	= &\, \int_\gw [\vp (u) u^{2p-1} + u^{2p-1} (v - w - J)]\, dx \\[2pt]
	= &\, \int_\gw [au^{2p+1} - bu^{2p+2} + u^{2p-1} (v - w - J)]\, dx.
	\end{split}
\eeq
On the right-hand side of the inequality \eqref{u2}, we have
\begin{gather*}
	a u^{2p+1} + u^{2p-1} (v - w + J) \leq \frac{b}{4} u^{2p+2} + C_b\left(a^{2p +2} + (|v(t)|^{\frac{2p+2}{3}} + |w(t)|^{\frac{2p+2}{3}}  + J^{\frac{2p+2}{3}} \right) \\[4pt]
	\leq \frac{1}{4} bu^{2p+2} + C_b\left(a^{2p +2} + C_{p,r} + J^{\frac{2p+2}{3}} \right) + \frac{1}{4} |v(t)|^{2p} + \frac{r}{4} |w(t)|^{2p}
\end{gather*}
where $C_b > 0$ is a constant only depending on the parameter $b$. Using the Young's inequality for each term on the left-hand side in the above inequality, the exponent 
$$
	\frac{2p+2}{3} \leq 2p \quad \text{for} \quad  p = 2, 3
$$
and  $C_{p,r} > 0$ is the constant only depending on $p$ and the parameter $r$, which is generated from using the following Young's inequality, 
\beq \bl{2p3}
	|v(t,x)|^{\frac{2p+2}{3}} + |w(t,x)|^{\frac{2p+2}{3}} \leq C_{p,r} + \frac{1}{4} |v(t,x)|^{2p} + \frac{r}{4} |w(t,x)|^{2p}.
\eeq
From \eqref{u2} and the above inequality it follows that 
\beq \label{u3}
	\begin{split}
	&\frac{1}{2p} \frac{d}{dt} \|u(t)\|^{2p}_{L^{2p}} + d_1 (2p-1) \|u^{p-1} \nb u \|^2_{L^2} \\[3pt]
	= &\, \int_\gw [au^{2p+1} - bu^{2p+2} + u^{2p-1} (v - w - J)]\, dx \\
	\leq &\, C_b\left(a^{2p +2} + C_{p,r} + J^{\frac{2p+2}{3}} \right) | \gw | - \frac{3}{4} \int_\gw b u^{2p+2} \, dx + \frac{1}{4} \int_\gw (v^{2p} + r w^{2p})\, dx.
	\end{split}
\eeq
The main controlling term is $- (3/4) \int_\gw b u^{2p+2} \, dx$.

\vspace{3pt}
Then taking the $L^2$ inner-product $\inpt{\eqref{veq}, v^{2p-1}}$, we obtain
\beq \label{v2}
	\begin{split}
	&\frac{1}{2p} \frac{d}{dt} \|v(t)\|^{2p}_{L^{2p}} + d_2 (2p-1) \|v^{p-1} \nb v \|^2_{L^2} \\
	= &\, \frac{1}{2p} \frac{d}{dt} \|v(t)\|^{2p}_{L^{2p}} + \frac{1}{p^2} d_2 (2p - 1) \| \nb (v^p)\|^2_{L^2} \\
	= &\, \int_\gw [\psi (u) v^{2p-1} - v^{2p}]\, dx =  \int_\gw [\alpha v^{2p-1} - \beta u^2 v^{2p-1} - v^{2p}]\, dx.
	\end{split}
\eeq
It is challenging to control the middle term on the right-hand side,
\beq \bl{uv}
	- \int_\gw \beta u^2 v^{2p-1} \, dx.
\eeq
We shall exploit the Gagliardo-Nirenberg inequality \eqref{GN} and the absorbing result in Theorem \ref{Th2} to handle this issue for $p = 2$ and $p = 3$ respectively. First use the Young's inequality to \eqref{Hld} and get
\beq \bl{Huv}
	\begin{split}
	- \int_\gw \beta u^2 v^{2p-1} \, dx &\leq \frac{1}{8} \int_\gw b u^{2p+2}\, dx + C_{b, \beta} \int_\gw v^{(2p-1)(1 + \frac{1}{p})}\, dx \\
	& = \frac{1}{8} \int_\gw b u^{2p+2}\, dx + C_{b, \beta} \int_\gw v^{2p + 1 - \frac{1}{p}}\, dx
	\end{split}
\eeq

Step 2. Prove \eqref{Kp} for $p = 2$.

For $p = 2$, the second term on the right-hand side of \eqref{Huv} is 
\beq \bl{Cv}
	C_{b, \beta} \int_\gw v^{5 - \frac{1}{2}}\, dx \leq \ve_1 \int_\gw v^{5}\, dx + C_{b, \beta, \ve_1} |\gw |,
\eeq
where $\ve_1 > 0$ is a small constant to be chosen and the constant $C_{b, \beta, \ve_1}$ depends on $C_{b, \beta}$ and $\ve_1$. Apply the Gagliardo-Nirenberg inequality \eqref{GN} to the interpolation of spaces
\beq \bl{L5A}
	L^1 (\gw) \hookrightarrow L^{2.5} (\gw) \hookrightarrow H^1 (\gw)
\eeq
to see there exists a constant $\eta_1 > 0$ such that
\beq \bl{L5B}
	 \int_\gw v^5 \, dx = \|v^2 \|^{2.5}_{L^{2.5}} \leq \eta_1 \left(\|\nb (v^2)\|^\theta  \, \|v^2 \|_{L^1}^{1-\theta} \right)^{2.5}
\eeq
where
$$
	- \frac{3}{2.5} = \theta \left(1- \frac{3}{2}\right) + (1 - \theta) (-3) \quad \text{so that} \quad \theta = \frac{18}{25}, \;\; 1 - \theta = \frac{7}{25}.
$$
Then for any weak solution $g(t) = (u(t), v(t), w(t))$ of the HIndmarsh-Rose evolutionary equation \eqref{pb} with any initial state $g_0 \in H$, there is a finite time $T(g_0)$ such that
\beq \bl{v5}
	\begin{split}
	 \int_\gw v^5\, dx &\leq \eta_1 \|\nb (v^2)\|^{18/10}  \, \|v^2 \|_{L^1}^{7/10} = \eta_1 \|\nb (v^2)\|^{9/5}  \, \|v^2 \|_{L^1}^{7/10} \\
	 & \leq \eta_1 \|\nb (v^2)\|^2 + \eta_1 \|v^2 \|_{L^1}^{7} \leq  \eta_1 \|\nb (v^2)\|^2 + \eta_1 K_1^7, \quad t \geq T(g_0),
	 \end{split}
\eeq
where the constant $K_1$ is given in \eqref{abs} and the inequality \eqref{v5} follows from Theorem \ref{Th2}. Substitute \eqref{v5} into \eqref{Cv} and then into \eqref{Huv}. We obtain
\beq \bl{CS1}
	\begin{split}
	- \int_\gw \beta u^2 v^{2p-1} \, dx & \leq \frac{1}{8} \int_\gw b u^{2p+2}\, dx + \ve_1 \int_\gw v^{5}\, dx + C_{b, \beta, \,\ve_1} |\gw | \\
	& \leq \frac{1}{8} \int_\gw b u^{2p+2}\, dx + \ve_1 ( \eta_1 \|\nb (v^2)\|^2 + \eta_1 K_1^7) + C_{b, \beta, \ve_1} |\gw | .
	\end{split}
\eeq
Now we choose 
$$
	\ve_1 = \frac{1}{2 \eta_1 p^2} d_2 (2p - 1) = \frac{3\,d_2}{8\,\eta_1}, \quad \textup{for} \;\,p = 2,
$$
so that
\beq \bl{Tm}
	 \ve_1  \eta_1 \|\nb (v^2)\|^2 \leq  \frac{1}{2p^2} d_2 (2p - 1) \| \nb (v^p)\|^2_{L^2}.
\eeq
Put together \eqref{v2} for $p=2$ with the estimates \eqref{CS1} and \eqref{Tm}. We obtain
\beq \label{vp2}
	\begin{split}
	&\frac{1}{2p} \frac{d}{dt} \|v(t)\|^{2p}_{L^{2p}} + \frac{1}{2} d_2 (2p-1) \|v^{p-1} \nb v \|^2_{L^2} \\[3pt]
	\leq &\,  \int_\gw [\alpha v^{2p-1}- v^{2p}]\, dx + \frac{1}{8} \int_\gw b u^{2p+2}\, dx + \ve_1 \eta_1 K_1^7 + C_{b, \beta, \ve_1} |\gw | \\[3pt]
	\leq & \, \frac{1}{8} \int_\gw b u^{2p+2}\, dx - \frac{1}{2} \int_\gw v^{2p}\, dx + \ve_1 \eta_1 K_1^7 + (C_2(\alpha) \alpha^{2p} + C_{b, \beta, \ve_1}) |\gw |
	\end{split}
\eeq
for $t \geq T(g_0)$, where $\alpha v^{2p-1} \leq \frac{1}{2} v^{2p} + C_2(\alpha) \alpha^{2p}$ for $p=2$ and $C_2(\alpha) > 0$ is a constant. 

Next take the $L^2$ inner-product $\inpt{\eqref{weq}, w^{2p-1}}$ to obtain
\beq \label{w2}
	\begin{split}
	&\frac{1}{2p} \frac{d}{dt} \|w(t)\|^{2p}_{L^{2p}} + d_3 (2p-1) \|w^{p-1} \nb w \|^2_{L^2} \\[5pt]
	= &\, \frac{1}{2p} \frac{d}{dt} \|w(t)\|^{2p}_{L^{2p}} + \frac{1}{p^2} d_3 (2p - 1) \| \nb (w^p)\|^2_{L^2} \\[5pt]
	= &\, \int_\gw [q (u - c) w^{2p-1} - r w^{2p}]\, dx \\
	\leq &\, \int_\gw \left(C |q (u - c)|^{2p} + \frac{1}{4}\, r w^{2p} - r w^{2p}\right) dx \\
	\leq &\, \int_\gw \left(C (u^{2p} + c^{2p}) - \frac{3}{4}\, r w^{2p}\right) dx \\
	\leq &\, \int_\gw \left(C_3 + \frac{1}{8}\, bu^{2p+2} + C_3 \, c^{2p} - \frac{3}{4}\, r w^{2p}\right) dx
	\end{split}
\eeq
where $C > 0$ is a generic constant from the expansion of the term $|q (u - c)|^{2p}$, and the constant $C_3 > 0$ is from the Young's inequality to raise the power $u^{2p} = u^4$ to $u^{2p+2} = u^6$ in the last step.

Next we sum up the estimates \eqref{u3}, \eqref{vp2} and \eqref{w2} for $p = 2$:
\begin{equation*}
	\begin{split}
	&\frac{1}{2p} \frac{d}{dt} (\|u(t)\|^{2p}_{L^{2p}} + \|v(t)\|^{2p}_{L^{2p}} + \|w(t)\|^{2p}_{L^{2p}}) \\[10pt]
	 & + \frac{1}{2} (2p-1) (d_1 \|u^{p-1} \nb u \|^2_{L^2} + d_2 \|v^{p-1} \nb v \|^2_{L^2} + d_3 \|w^{p-1} \nb w \|^2_{L^2}) \\[8pt]
	\leq &\, C_b\left(a^{2p +2} + C_{p,r} + J^{\frac{2p+2}{3}} \right) | \gw | - \frac{3}{4} \int_\gw b u^{2p+2} \, dx + \frac{1}{4} \int_\gw (v^{2p} + r w^{2p})\, dx \\
	& + \frac{1}{8} \int_\gw b u^{2p+2}\, dx - \frac{1}{2} \int_\gw v^{2p}\, dx + \ve_1 \eta_1 K_1^7 + (C_2 (\alpha) \alpha^{2p} + C_{b, \beta, \ve_1}) |\gw |  \\
	& + \int_\gw \left(C_3 + \frac{1}{8}\, bu^{2p+2} + C_3 \, c^{2p} - \frac{3}{4}\, r w^{2p}\right) dx \\[6pt]
	\leq &\, C_b\left(a^{2p +2} + C_{p,r} + J^{\frac{2p+2}{3}} \right) | \gw | + (C_2 (\alpha) \alpha^{2p} + C_{b, \beta, \ve_1} + C_3 (1 + c^{2p})) |\gw | \\
	& + \ve_1 \eta_1 K_1^7 - \left(\int_\gw \frac{1}{2} b u^{2p+2} \, dx +  \frac{1}{4} \int_\gw v^{2p}\, dx + \frac{1}{2} \int_\gw r w^{2p} \right) dx, \;\; \text{for} \;\; t \geq T(g_0).
	\end{split}
\end{equation*}
It follows that, for $p=2$,
\beq \label{G2}
	\begin{split}
	&\frac{1}{2p} \frac{d}{dt} (\|u(t)\|^{2p}_{L^{2p}} + \|v(t)\|^{2p}_{L^{2p}} + \|w(t)\|^{2p}_{L^{2p}}) \\[2pt]
	 & + \frac{1}{2} (2p-1) (d_1 \|u^{p-1} \nb u \|^2_{L^2} + d_2 \|v^{p-1} \nb v \|^2_{L^2} + d_3 \|w^{p-1} \nb w \|^2_{L^2}) \\[2pt]
	\leq &\, C_4 |\gw| + \ve_1 \eta_1 K_1^7 - \frac{1}{4} \left(\int_\gw b u^{2p+2} \, dx +  \int_\gw v^{2p}\, dx +  \int_\gw r w^{2p} \right) dx \\
	\leq &\, C_4 |\gw| + \ve_1 \eta_1 K_1^7 - \frac{1}{4} \left(\int_\gw b u^{2p} \, dx - \frac{b}{p+1} +  \int_\gw v^{2p}\, dx +  \int_\gw r w^{2p} \right) dx \\
	\leq &\, \left(C_4 + \frac{b}{4(p+1)}\right) |\gw| + \ve_1 \eta_1 K_1^7  \\
	 - &\,\frac{1}{4} \left(\int_\gw b u^{2p} \, dx + \int_\gw v^{2p}\, dx +  \int_\gw r w^{2p} \right) dx, \quad \text{for} \;\; t \geq T(g_0), \; g_0 \in H,
	\end{split}
\eeq 
where 
$$
        C_4 = C_b\left(a^{2p +2} + C_{p,r} + J^{\frac{2p+2}{3}} \right) + C_2 (\alpha) \alpha^{2p} + C_{b, \beta, \ve_1} + C_3 (1 + c^{2p}).
$$
Then we can apply the Gronwall inequality to the following differential inequality reduced from \eqref{G2} by moving the three integral terms to the left-hand side,

\beq \label{L2p}
	\begin{split}
	&\frac{d}{dt} (\|u(t)\|^{4}_{L^{4}} + \|v(t)\|^{4}_{L^{4}} + \|w(t)\|^{4}_{L^{4}}) \\[4pt]
	+ &\, \min \{b, r, 1\} (\|u(t)\|^{4}_{L^{4}} + \|v(t)\|^{4}_{L^{4}} + \|w(t)\|^{4}_{L^{4}}) \\
	\leq &\, 4 \left(C_4 + \frac{b}{12}\right) |\gw| + 4 \ve_1 \eta_1 K_1^7, \quad \text{for} \;\; t \geq T(g_0).
	\end{split}
\eeq
Hence we obtain the bounded estimate
\beq \label{Lsp}
	\|(u(t), v(t), w(t))\|_{L^{4}}^{4} \leq e^{-\lambda (t - T(g_0))} \|g_0\|^4_{L^4} + M_2 |\gw |  + 4 \lambda^{-1} \ve_1 \eta_1 K_1^7,  \;\, t \geq T(g_0),  g_0 \in H,
\eeq
where $K_1$ is shown in \eqref{abs} and 
$$
	\lambda = \min \{b, r, 1\} \quad \text{and} \quad M_2 = \frac{1}{\lambda} \left(4C_4 +  \frac{b}{3}\right). 
$$
Let $t \to \infty$ in \eqref{Lsp}.Then the absorbing property \eqref{Kp} is proved for $p = 2$, 
\beq \bl{vK2}
	\limsup_{t \to \infty}\, (\|u(t)\|^4_{L^4} + \|v(t)\|^4_{L^4} + \|w(t)\|^4_{L^4}) < K_2
\eeq
and
\beq \bl{K2}
	K_2  = M_2 |\gw| + 4\lambda^{-1} \ve_1 \eta_1 K_1^7 + 1.     
\eeq

Step 3. Prove \eqref{Kp} for $p = 3$ by means of bootstrap argument.

For $p = 3$, the second term on the right-hand side of \eqref{Huv} is 
\beq \bl{Cv3}
	C_{b, \beta} \int_\gw v^{7 - \frac{1}{3}}\, dx \leq \ve_2 \int_\gw v^{7}\, dx + C_{b, \beta, \ve_2} |\gw |,
\eeq
where $\ve_2 > 0$ is a small constant to be chosen and the constant $C_{b, \beta, \ve_2}$ depends on $C_{b, \beta}$ and $\ve_2$. Similarly we use the Gagliardo-Nirenberg inequality \eqref{GN} for spaces
\beq \bl{L7A}
	L^1 (\gw) \hookrightarrow L^{7/3} (\gw) \hookrightarrow H^1 (\gw)
\eeq
to claim that there exists a constant $\eta_2 > 0$ such that
\beq \bl{L7B}
	 \int_\gw v^7 \, dx = \|v^3 \|^{7/3}_{L^{7/3}} \leq \eta_2 \left(\|\nb (v^3)\|^\theta  \, \|v^3 \|_{L^1}^{1-\theta} \right)^{7/3}
\eeq
where
$$
	- \frac{3}{7/3} = \theta \left(1- \frac{3}{2}\right) + (1 - \theta) (-3) \quad \text{so that} \quad \theta = \frac{24}{35}, \;\; 1 - \theta = \frac{11}{35}.
$$
Then
\beq \bl{v7}
	 \int_\gw v^7\, dx \leq \eta_2 \|\nb (v^3)\|^{8/5}  \, \|v^3 \|_{L^1}^{11/15} \leq \eta_2 \|\nb (v^3)\|^2 + \eta_2 \|v^3 \|_{L^1}^{11/3}.
\eeq
Substitute \eqref{v7} into \eqref{Cv3} and then into \eqref{Huv}. We obtain
\beq \bl{CS2}
	\begin{split}
	- &\int_\gw \beta u^2 v^{2p-1} \, dx  \leq \frac{1}{8} \int_\gw b u^{2p+2}\, dx + \ve_2 \int_\gw v^{7}\, dx + C_{b, \beta, \,\ve_2} |\gw | \\[5pt]
	 \leq &\, \frac{1}{8} \int_\gw b u^{2p+2}\, dx + \ve_2 ( \eta_2 \|\nb (v^3)\|^2 + \eta_2 \|v^3 \|_{L^1}^{11/3}) + C_{b, \beta, \ve_2} |\gw | .
	\end{split}
\eeq
Now we choose 
$$
	\ve_2 = \frac{1}{2 \eta_2 \,p^2} d_2 (2p - 1) = \frac{5\,d_2}{18\,\eta_2}, \quad \textup{for} \;\, p = 3,
$$
so that
\beq \bl{Tm3}
	 \ve_2  \eta_2 \|\nb (v^3)\|^2 \leq  \frac{1}{2p^2} d_2 (2p - 1) \| \nb (v^p)\|^2_{L^2}.
\eeq
Put together \eqref{v2} for $p=3$ with the estimates \eqref{CS2} and \eqref{Tm3}. We obtain
\beq \label{vp3}
	\begin{split}
	&\frac{1}{2p} \frac{d}{dt} \|v(t)\|^{2p}_{L^{2p}} + \frac{1}{2} d_2 (2p-1) \|v^{p-1} \nb v \|^2_{L^2} \\
	\leq &\,  \int_\gw [\alpha v^{2p-1}- v^{2p}]\, dx + \frac{1}{8} \int_\gw b u^{2p+2}\, dx + \ve_2 \eta_2 \|v^3 \|_{L^1}^{11/3} + C_{b, \beta, \ve_2} |\gw | \\[3pt]
	\leq &\,  \frac{1}{8} \int_\gw b u^{2p+2}\, dx - \frac{1}{2} \int_\gw v^{2p}\, dx + (C_5(\alpha) \alpha^{2p} + C_{b, \beta, \ve_2}) |\gw | + \frac{1}{2}\ve_2 \eta_2 \left(\|v \|^2_{L^2} + \|v\|^4_{L^4}\right)^{11/3} 
	\end{split}
\eeq
where we used $\alpha v^{2p-1} \leq \frac{1}{2} v^{2p} + C_5 (\alpha) \alpha^{2p}$ for $p=3$, $C_5 (\alpha) > 0$ is a constant, and $\|v^3 \|_{L^1} \leq \frac{1}{2} (\|v \|^2_{L^2} + \|v\|^4_{L^4})$. 

By the absorbing property \eqref{Kp} for $p=2$, which was proved in Step 2, we have 
$$
	\limsup_{t \to \infty} \|v(t)\|^4_{L^4} < K_2
$$
where $K_2$ is in \eqref{K2}. Thus there exists a finite time $\widetilde{T}(g_0) \geq T(g_0) > 0$ such that, for $t \geq \widetilde{T}(g_0)$,
\beq \label{v3}
	\begin{split}
	&\frac{1}{2p} \frac{d}{dt} \|v(t)\|^{2p}_{L^{2p}} + \frac{1}{2} d_2 (2p-1) \|v^{p-1} \nb v \|^2_{L^2} \\[3pt]
	\leq &\, \frac{1}{8} \int_\gw b u^{2p+2}\, dx - \frac{1}{2} \int_\gw v^{2p}\, dx + (C_5(\alpha) \alpha^{2p} + C_{b, \beta, \ve_2}) |\gw | + \frac{1}{2}\ve_2 \eta_2 \left(\|v \|^2_{L^2} + \|v\|^4_{L^4}\right)^{11/3} \\
	\leq &\, \frac{1}{8} \int_\gw b u^{2p+2}\, dx - \frac{1}{2} \int_\gw v^{2p}\, dx + (C_5(\alpha) \alpha^{2p} + C_{b, \beta, \ve_2}) |\gw | + \ve_2 \eta_2 (K_1 + K_2)^{11/3}.
	\end{split}
\eeq
Sum up the same \eqref{u3}, \eqref{w2} and the new estimate \eqref{v3} for the second component $v(t, x)$ with $p=3$. Then we have
\begin{equation*}
	\begin{split}
	&\frac{1}{2p} \frac{d}{dt} (\|u(t)\|^{2p}_{L^{2p}} + \|v(t)\|^{2p}_{L^{2p}} + \|w(t)\|^{2p}_{L^{2p}}) \\[3pt]
	 & + \frac{1}{2} (2p-1) (d_1 \|u^{p-1} \nb u \|^2_{L^2} + d_2 \|v^{p-1} \nb v \|^2_{L^2} + d_3 \|w^{p-1} \nb w \|^2_{L^2}) \\[3pt]
	\leq &\, C_b\left(a^{2p +2} + C_{p,r} + J^{\frac{2p+2}{3}} \right) | \gw | - \frac{3}{4} \int_\gw b u^{2p+2} \, dx + \frac{1}{4} \int_\gw (v^{2p} + r w^{2p})\, dx \\
	& + \frac{1}{8} \int_\gw b u^{2p+2}\, dx - \frac{1}{2} \int_\gw v^{2p}\, dx + (C_5(\alpha) \alpha^{2p} + C_{b, \beta, \ve_2}) |\gw | + \ve_2 \eta_2 (K_1 + K_2)^{11/3} \\
	& + \int_\gw \left(C_5 (\alpha) + \frac{1}{8}\, bu^{2p+2} + C_5 (\alpha) c^{2p} - \frac{3}{4}\, r w^{2p}\right) dx \\[6pt]
	\leq &\, C_b\left(a^{2p +2} + C_{p,r} + J^{\frac{2p+2}{3}} \right) | \gw | + (C_5(\alpha) \alpha^{2p} + C_{b, \beta, \ve_1} + C_5(\alpha) (1 + c^{2p})) |\gw | \\
	& + \ve_2 \eta_2 (K_1 + K_2)^{11/3} - \left(\int_\gw \frac{1}{2} b u^{2p+2} \, dx +  \frac{1}{4} \int_\gw v^{2p}\, dx + \frac{1}{2} \int_\gw r w^{2p} \right) dx,
	\end{split}
\end{equation*}
for $t \geq \widetilde{T}(g_0)$. It follows that
\begin{equation*}
	\begin{split}
	&\frac{1}{2p} \frac{d}{dt} (\|u(t)\|^{2p}_{L^{2p}} + \|v(t)\|^{2p}_{L^{2p}} + \|w(t)\|^{2p}_{L^{2p}}) \\
	 & + \frac{1}{2} (2p-1) (d_1 \|u^{p-1} \nb u \|^2_{L^2} + d_2 \|v^{p-1} \nb v \|^2_{L^2} + d_3 \|w^{p-1} \nb w \|^2_{L^2}) \\
	\leq &\, C_6 |\gw| + \ve_2 \eta_2 (K_1 + K_2)^{11/3} - \frac{1}{4} \left(\int_\gw b u^{2p+2} \, dx +  \int_\gw v^{2p}\, dx +  \int_\gw r w^{2p} \right) dx\\[3pt]
	\leq &\, C_6 |\gw| + \ve_2 \eta_2 (K_1 + K_2)^{11/3} \\
	- &\, \frac{1}{4} \left(\int_\gw b u^{2p} \, dx - \frac{b}{p+1} +  \int_\gw v^{2p}\, dx +  \int_\gw r w^{2p} \right) dx \\
	\leq &\, \left(C_6 + \frac{b}{16}\right) |\gw| + \ve_2 \eta_2 (K_1 + K_2)^{11/3} - \frac{1}{4} \left(\int_\gw b u^{2p} \, dx + \int_\gw v^{2p}\, dx +  \int_\gw r w^{2p} \right) dx,
	\end{split}
\end{equation*}
for $t \geq \widetilde{T} (g_0), \, g_0 \in H$, where
$$
        C_6 = C_b\left(a^{2p +2} + C_{p,r} + J^{\frac{2p+2}{3}} \right) + C_5 (\alpha) \alpha^{2p} + C_{b, \beta, \ve_1} + C_3 (1 + c^{2p}).
$$
Similar to \eqref{L2p}, for $p=3$ we get the differential inequality
\beq \label{L3p}
	\begin{split}
	&\frac{d}{dt} (\|u(t)\|^{2p}_{L^{2p}} + \|v(t)\|^{2p}_{L^{2p}} + \|w(t)\|^{2p}_{L^{2p}}) \\[4pt]
	+ &\, \frac{p}{2} \min \{b, r, 1\} (\|u(t)\|^{2p}_{L^{2p}} + \|v(t)\|^{2p}_{L^{2p}} + \|w(t)\|^{2p}_{L^{2p}}) \\
	\leq &\, \left( 6C_6 + \frac{3b}{8}\right) |\gw| + 6\, \ve_2 \eta_2 (K_1 + K_2)^{11/3}, \;\; \text{for} \;\; t \geq \widetilde{T} (g_0).
	\end{split}
\eeq
Apply the Gronwall inequality to \eqref{L3p} and yield
\beq \bl{g3}
	\|(u(t), v(t), w(t))\|_{L^{6}}^{6} \leq e^{-\lambda (t - T(g_0))} \|g_o \|^6_{L^6} + M_3 |\gw | + 6\lambda^{-1} \ve_2 \eta_2 (K_1 + K_2)^{11/3}
\eeq
for $t \geq \widetilde{T} (g_0),  g_0 \in H$, and here
$$
	\lambda = \frac{3}{2}\min \{b, r, 1\} \quad \text{and} \quad M_3 = \frac{1}{\lambda} \left(6\,C_6 + \frac{3b}{8}\right). 
$$
Let $t \to \infty$. Then the absorbing property \eqref{Kp} is proved for the case $p = 3$, namely,
$$
	\limsup_{t \to \infty}\, (\|u(t)\|^6_{L^6} + \|v(t)\|^6_{L^6} + \|w(t)\|^6_{L^6}) < K_3
$$
and
\beq \bl{K3}
	K_3 = M_3 \gw | + 6\lambda^{-1} \ve_2 \eta_2 (K_1 + K_2)^{11/3} + 1.
\eeq
The proof is completed.
\end{proof}

\section{\textbf{Asymptotic Compactness and Global Attractor}}

In this section, we show that the Hindmarsh-Rose semiflow $\{S(t)\}_{t \geq 0}$ is asymptotically compact and then reach the main result on the existence of a global attractor for this dynamical system associated with the diffusive Hindmarsh-Rose equations.

\begin{theorem} \label{Th3}
	For any given bounded set $B \in H$, there exists a finite time $T_1 (B) > 0$ such that for any initial state $g_0 = (u_0, v_0, w_0) \in B$, the weak solution $g(t) = S(t)g_0 = (u(t), v(t), w(t))$ of the initial value problem \eqref{pb} satisfies 
\beq \label{ac}
	\| (u(t), v(t), w(t))\|_E^2 \leq Q_1, \quad \text {for} \;\; t \geq T_1 (B)
\eeq
where $Q_1 > 0$ is a constant depending only on $K_1$ given in \eqref{abs} and $|\gw|$, and $T_1 (B) > 0$ only depends on the bounded set $B$.
\end{theorem}

\begin{proof}
Take the $L^2$ inner-product $\inpt{\eqref{ueq}, -\gd u(t)}$ to obtain

\begin{equation*}
\begin{split}
	 &\frac{1}{2} \frac{d}{dt} \|\nb u\|^2 + d_1 \|\gd u\|^2 = \int_\gw (- a u^2 \gd u - 3 b u^2 |\nb u|^2 -v \gd u + w \gd u - J \gd u)\, dx \\
	 \leq &\, \int_\gw \left(\frac{2 v^2}{d_1} + \frac{d_1}{8} |\gd u|^2 + \frac{2 w^2}{d_1} + \frac{d_1}{8} |\gd u|^2 + \frac{2 J^2}{d_1} + \frac{d_1}{8} |\gd u|^2 + \frac{2 a^2 u^4}{d_1} + \frac{d_1}{8} |\gd u|^2\right)\, dx \\
	 - &\, \int_\gw 3bu^2 |\nb u|^2\,dx.
\end{split}
\end{equation*}
It follows that
\beq \label{nbu}
	\frac{d}{dt} \|\nb u\|^2 + d_1 \|\gd u\|^2 + 6 b \|u \nb u\|^2 \leq \frac{4}{d_1} \|v\|^2 + \frac{4}{d_1} \|w\|^2 + \frac{4 J^2}{d_1} |\gw| + \frac{4 a^2}{d_1} \|u\|^4_{L^4}.
\eeq
Next take the $L^2$ inner-product $\inpt{\eqref{veq}, -\gd v(t)}$ to get
\begin{equation*}
\begin{split}
	&\frac{1}{2} \frac{d}{dt} \|\nb v\|^2 + d_2 \|\gd v\|^2 =  \int_\gw (-\alpha \gd v + \beta u^2 \gd v - |\nb v|^2)\, dx\\
	\leq &\, \int_\gw \left(\frac{\alpha^2}{d_2} + \frac{d_2}{4} |\gd v|^2 + \frac{\beta^2 u^4}{d_2} + \frac{d_2}{4} |\gd v|^2\right)\, dx  - \|\nb v\|^2.
\end{split}
\end{equation*}
It follows that
\beq \label{nbv}
	\frac{d}{dt} \|\nb v\|^2 + d_2 \|\gd v\|^2 + 2 \|\nb v\|^2 \leq \frac{2 \alpha^2}{d_2} |\gw| + \frac{2 \beta^2}{d_2} \|u\|^4_{L_4}.
\eeq
Then taking the $L_2$ inner-product $\inpt{\eqref{weq}, -\gd w(t)}$, we get
\begin{equation*}
\begin{split}
	&\frac{1}{2} \frac{d}{dt} \|\nb w\|^2 + d_3 \|\gd w\|^2  = \int_\gw (q c \gd w - q u \gd w - r |\nb w|^2)\, dx\\
	\leq &\, \int_\gw \left(\frac{q^2 c^2}{d_3} + \frac{d_3}{4} |\gd w|^2 + \frac{q^2 u^2}{d_3} + \frac{d_3}{4} |\gd w|^2\right)\, dx - r \|\nb w\|^2.
\end{split}
\end{equation*} 
It follows that
\beq \label{nbw}
	\frac{d}{dt} \|\nb w\|^2 + d_3 \|\gd w\|^2 + 2 r \|\nb w\|^2 \leq \frac{2 q^2 c^2}{d_3} |\gw| + \frac{2 q^2}{d_3} \|u\|^2_{L_2}.
\eeq
Sum up the above estimates \eqref{nbu}, \eqref{nbv} and \eqref{nbw} to obtain
\beq \label{uu}
	\begin{split}
	&\frac{d}{dt} (\|\nb u\|^2 + \|\nb v\|^2 + \|\nb w\|^2) + d_1 \|\gd u\|^2 + d_2 \|\gd v\|^2 + d_3 \|\gd w\|^2\\[4pt] 
	& + 6 b \|u \nb u\|^2 + 2 \|\nb v\|^2 + 2 r \|\nb w\|^2\\[2pt]
	\leq &\, \frac{4}{d_1} \|v\|^2 + \frac{4}{d_1} \|w\|^2 + \frac{2 q^2}{d_3} \|u\|^2 + \left(\frac{4 a^2}{d_1} + \frac{2 \beta^2}{d_2}\right) \|u\|^4_{L_4} + \left(\frac{4 J^2}{d_1} + \frac{2 \alpha^2}{d_2} + \frac{2 q^2 c^2}{d_3}\right) |\gw|.
	\end{split}
\eeq
Since $H_1 (\gw) \hookrightarrow L_4 (\gw)$ is a continuous embedding, there is a positive constant $\eta > 0$ such that
$$
	\|u\|_{L_4} \leq \eta \|u\|_{H_1} \leq \eta \, \sqrt{\|u\|^2 + \|\nb u\|^2}.
$$
Here the Poincar\'e\xspace inequality is not valid due to the Neumann boundary condition \eqref{nbc}. Then we have
$$
\|u\|^4_{L_4} \leq \eta^4 (\|u\|^2 + \|\nb u\|^2)^2 \leq 2 \eta^4 \|u\|^4 + 2 \eta^4 \|\nb u\|^4.
$$

According to Theorem \ref{Th2} and \eqref{T0B}, there is a finite time $T_0 (B) > 0$ such that the solution $g(t) = (u(t), v(t), w(t))$ with any initial state $g_0 \in B$ will permanently enter the absorbing ball $B_0$ shown in \eqref{abs}. It implies that the sum of the  $L^2$-norms of all three components of the solution satisfies
\beq \label{uvwK}
	\|u(t)\|^2 + \|v(t)\|^2 + \|w(t)\|^2 \leq K_1, \quad \text{for any}\;  t > T_0 (B), \;  g_0 \in B.
\eeq
Then \eqref{uu} yields the following differential inequality
\beq \label{uu1}
	\begin{split}
	&\frac{d}{dt} (\|\nb u\|^2 + \|\nb v\|^2 + \|\nb w\|^2) + d_1 \|\gd u\|^2 + d_2 \|\gd v\|^2 + d_3 \|\gd w\|^2\\[5pt] 
	& + 6 b \|u \nb u\|^2 + 2 \|\nb v\|^2 + 2 r \|\nb w\|^2\\[4pt]
	\leq &\; \max \left\{\frac{4}{d_1} , \frac{2 q^2}{d_3}\right\} K_1 + \left(\frac{8 a^2}{d_1} + \frac{4 \beta^2}{d_2}\right) \eta^4 \|\nb u\|^4 \\
	& + \eta^4 K_1^2 \left(\frac{8 a^2}{d_1} + \frac{4 \beta^2}{d_2}\right) + \left(\frac{4 J^2}{d_1} + \frac{2 \alpha^2}{d_2} + \frac{2 q^2 c^2}{d_3}\right) |\gw|, \; \; t > T_0(B), \; g_0 \in B. \\
	\end{split}
\eeq
The inequality \eqref{uu1} implies that for any initial data $g_0 \in B$ it holds that
\beq \label{uu2}
	\begin{split}
	&\frac{d}{dt} \|(\nb u, \nb v, \nb w)\|^2\\[3pt]
	\leq &\; \eta^4 \left(\frac{8 a^2}{d_1} + \frac{4 \beta^2}{d_2}\right) \|(\nb u, \nb v, \nb w)\|^2\ \|(\nb u, \nb v, \nb w)\|^2\\
	& + \max \left\{\frac{4}{d_1} , \frac{2 q^2}{d_3}\right\} K_1 + \eta^4 K_1^2 \left(\frac{8 a^2}{d_1} + \frac{4 \beta^2}{d_2}\right) + \left(\frac{4 J^2}{d_1} + \frac{2 \alpha^2}{d_2} + \frac{2 q^2 c^2}{d_3}\right) |\gw|
	\end{split}
\eeq
for all $t > T_0 (B)$. 

Now we can apply the uniform Gronwall inequality \cite{SY} to the differential inequality \eqref{uu2}, which is written as
\beq \label{uu3}
	\frac{d}{dt} \sigma (t) \leq \rho (t) \, \sigma (t) + h(t), \quad \text{for} \; t > T_0 (B), \; g_0 \in B,
\eeq
where
\begin{gather*}
	\sigma(t) = \|(\nb u (t), \nb v(t), \nb w(t))\|^2, \\[4pt]
	\rho(t) = \eta^4 \left(\frac{8 a^2}{d_1} + \frac{4 \beta^2}{d_2}\right) \|(\nb u(t), \nb v(t), \nb w(t))\|^2,
\end{gather*}
and $h(t)$ is a constant
$$
	h(t) = \max \left\{\frac{4}{d_1} , \frac{2 q^2}{d_3}\right\} K_1 + \eta^4 K_1^2 \left(\frac{8 a^2}{d_1} + \frac{4 \beta^2}{d_2}\right) + \left(\frac{4 J^2}{d_1} + \frac{2 \alpha^2}{d_2} + \frac{2 q^2 c^2}{d_3}\right) |\gw|.
$$
For any $t > T_0 (B)$, integration of \eqref{E2} implies that
\begin{equation*}
	\begin{split}
	&\int_{t}^{t+1} \min \{d_1, d_2, d_3\} (C_1 \|\nb u(s)\|^2 + \|\nb v(s)\|^2+ \|\nb w(s)\|^2)\; ds\\[3pt]
	\leq & \, C_1 \|u(t)\|^2 + \|v(t)\|^2 + \|w(t)\|^2 + r_1 M |\gw| \leq \max \{1, C_1\}  K_1 +  r_1 M |\gw|, \;\; t > T_0 (B).
	\end{split}
\end{equation*}
Here the constant $M > 0$ is shown in \eqref{dse}. Thus we get
\beq \label{uu4}
	\int_{t}^{t+1} \sigma(s)\; ds \leq \frac{r_1 M |\gw| + \max \{1, C_1\} K_1}{\min \{d_1, d_2, d_3\}  \min \{1, C_1\}}  \quad \text{for}\; t > T_0 (B),\; g_0 \in B.
\eeq
Hence we also have
\beq \label{uu5}
	\int_{t}^{t+1} \rho(s)\; ds \leq \eta^4 \left(\frac{8 a^2}{d_1} + \frac{4 \beta^2}{d_2}\right) \left(\frac{r_1 M |\gw| + \max \{1, C_1\} K_1}{\min \{d_1, d_2, d_3\}  \min \{1, C_1\}}\right).
\eeq
Denote by
$$
	N = \eta^4 \left(\frac{8 a^2}{d_1} + \frac{4 \beta^2}{d_2}\right) \left(\frac{r_1 M |\gw| + \max \{1, C_1\} K_1}{\min \{d_1, d_2, d_3\}  \min \{1, C_1\}}\right).
$$
The uniform Gronwall inequality applied to \eqref{uu3} yields
\beq \label{uu6}
	\| (\nb u(t), \nb v(t), \nb w(t))\|^2 \leq C_7\, e^N, \quad \text{for any}. \; \; t \geq T_0(B) +1, \; g_0 \in B,
\eeq
where 
\begin{equation*}
	\begin{split}
	C_7 &\, = \frac{r_1 M |\gw| + \max \{1, C_1\} K_1}{\min \{d_1, d_2, d_3\}  \min \{1, C_1\}} + \max \left\{\frac{4}{d_1} , \frac{2 q^2}{d_3}\right\} K_1 \\[3pt]
	&\, + \eta^4 K_1^2 \left(\frac{8 a^2}{d_1} + \frac{4 \beta^2}{d_2}\right) + \left(\frac{4 J^2}{d_1} + \frac{2 \alpha^2}{d_2} + \frac{2 q^2 c^2}{d_3}\right) |\gw|. 
	\end{split}
\end{equation*}
Finally, we complete the proof of \eqref{ac}: 
$$
	\| (u(t), v(t), w(t))\|_E^2 = \| (u, v, w)\|^2+ \| \nb (u, v, w)\|^2 \leq Q_1 = K_1 + C_7\, e^N
$$
for $t \geq T_1 (B) = T_0 (B) + 1$. The proof is completed.
\end{proof}

We now prove the main result on the existence of a global attractor for the diffusive Hindmarsh-Rose semiflow $\{S(t)\}_{t \geq 0}$.

\begin{theorem}[\textbf{The Existence of Global Attractor}] \label{MTh}
	For any positive parameters $d_1, d_2, d_3, a, b, \alpha, \beta, q, r, J$ and $c \in \mathbb{R}$, there exists a global attractor $\mathscr{A}$ in the space $H = L^2 (\gw, \mathbb{R}^3)$ for the Hindmarsh-Rose semiflow $\{S(t)\}_{t \geq 0}$ generated by the weak solutions of the diffusive Hindmarsh-Rose equations \eqref{pb}. Moreover, the global attractor $\mathscr{A}$ is an $(H, E)$-global attractor.
\end{theorem}

\begin{proof}
	In Theorem \ref{Th2} it has been proved that there is an absorbing set $B_0 \in H$ for the Hindmarsh-Rose semiflow $\{S(t)\}_{t \geq 0}$. In Theorem \ref{Th3}, it is shown that for any given bounded set $B \subset H$,
$$
	\|S(t) g_0\|^2_E \leq Q_1, \quad \text{for}\; \, t \geq T_1 (B)\; \,\text{and all}\; \, g_0 \in B.
$$
This implies that $\bigcup_{t \geq T_1 (B)} S(t) B$ is a bounded set in $E$, and therefore a precompact set in $H$ due to the compact embedding $E \hookrightarrow H$. 
Therefore, the Hindmarsh-Rose semiflow $\{S(t)\}_{t \geq 0}$ is asymptotically compact in $H$. Since the two conditions in Proposition \eqref{L:basic} are satisfied, we conclude that there exists a global attractor $\mathscr{A}$ in the phase space $H$ for this Hindmarsh-Rose semiflow $\{S(t)\}_{t \geq 0}$ and
\beq \bl{msA}
	 \ms{A} = \bigcap_{\tau \geq 0} \; \overline{\bigcup_{t \geq \tau} (S(t)B_0)}.
\eeq

Next we prove that this global attractor $\ms{A}$ is a bi-space $(H, E)$-global attractor. Actually Theorem \ref{Th3} shows that there is a bounded absorbing set $B_1 \subset E$ with respect to the $E$-norm for the Hindmarsh-Rose semiflow $\{S(t)\}_{t \geq 0}$ on $H$. Indeed,
\beq \bl{B1}
	B_1 = \{g \in E : \|g\|^2_E \leq Q_1\}.
\eeq

Moreover, we can show that the Hindmarsh-Rose semiflow $\{S(t)\}_{t \geq 0}$ is asymptotically compact not only in $H$ but also in the space $E$ with respect to the $E$-norm. Let $T > 0$ be arbitrarily given. For any time sequence $\{t_n\}^{\infty}_{n=1},\; t_n \rightarrow \infty$, and any bounded sequence $\{g_n\} \in E$, there is an integer $n_0 \geq 1$ such that $t_n > T$ for all $n > n_0$. By Theorem \ref{Th3} and the bounded sequence $\{g_n\}$ in $E$, we have
$$
	\{S(t_n - T) g_n\}_{n > n_0} \; \text{is bounded set in $E$}.
$$ 
Since $E$ is a Hilbert space, there exists an increasing subsequence of integers $\{n_i\}^{\infty}_{i=1}$ where $n_i > n_0$ such that the following weak limit exists,
$$
	(w) \lim_{i \to \infty} S(t_{n_i} - T) g_{n_i} = g^* \in E.
$$
Since $E$ is compactly embedded in $H$, we can take subsequence of $\{n_i\}^{\infty}_{i=1}$ and relabel it as the same as $\{n_i\}^{\infty}_{i=1}$, such that the following strong convergence holds,
$$
	(s) \lim_{i \to \infty} S(t_{n_i} - T) g_{n_i} = g^* \in H.
$$
Hence the following strong convergence in $E$ holds,
$$
	\lim_{i \to \infty} S(t_{n_i}) g_{n_i} = \lim_{i \to \infty} S(T) S(t_{n_i} - T) g_{n_i} = S(T) g^* \in E.
$$
This proves that $\{S(t)\}_{t \geq 0}$ is asymptotically compact in $E$. 

Thus, by Proposition \eqref{L:basic}, there exists a global attractor $\mathscr{A}_E$ in $E$ for the semiflow $\{S(t)\}_{t \geq 0}$. Since $\ms{A}_E$ attracts the set $B_1$ in the $E$-norm and, on the other hand, Theorem \ref{Th3} shows that $B_1$ absorbs any bounded subset $B$ of $H$, then the global attractor $\mathscr{A}_E$ attracts any given bounded set $B \subset H$ in $E$-norm. Therefore, $\ms{A}_E$ is an ($H, E$) global attractor. 

Finally, since $\mathscr{A}$ is bounded and invariant in $H$ and in $E$, it holds that
\begin{equation*}
	\begin{split}
	&\mathscr{A}_E \; \text{attracts} \; \mathscr{A} \; \text{in $E$, so that} \; \mathscr{A} \subset \mathscr{A}_E , \\
	&\mathscr{A} \; \text{attracts} \; \mathscr{A}_E \; \text{in $H$, so that} \; \mathscr{A}_E \subset \mathscr{A}.
	\end{split}
\end{equation*}
It concludes that $\mathscr{A}$ = $\mathscr{A}_E$. Therefore, the global attractor $\mathscr{A}$ in $H$ is an ($H, E$) global attractor for the Hindmarsh-Rose semiflow.
\end{proof}

\section{\textbf{Regularity Properties and Structure of the Global Attractor}}	

In this section, we shall prove the regularity properties of the global attractor $\mathscr{A}$ in the spaces $ L^\infty (\gw, \mathbb{R}^3)$ and $H^2 (\gw, \mathbb{R}^3)$. And we shall show that the Hindmarsh-Rose semiflow is a gradient system so that its global attractor turns out to be the union of the unstable manifolds of all the steady states. 
\begin{theorem}\label{p1a}
	The global attractor $\mathscr{A}$ for the Hindmarsh-Rose semiflow $\{S(t)\}_{t \geq 0}$ in the space $H$ is a bounded set in $L^\infty (\gw, \mathbb{R}^3)$. There is a constant $C_\infty > 0$ such that 
	\beq \label{uu7}
	\sup_{g \in \mathscr{A}}\, \|g\|_{L^{\infty}} \leq C_\infty.  
	\eeq
\end{theorem}

\begin{proof}
The analytic $C_0$-semigroup $\{e^{A t}\}_{t \geq 0}$ has the regularity property \cite{SY} that $e^{A t}: L^p(\gw) \rightarrow L^{\infty}(\gw)$ for $p \geq 1, t > 0$, and there is a constant $c(p) > 0$ such that
\begin{equation}\label{p1b}
	\|e^{A t}\|_{\mathcal{L} {(L^p, L^\infty)}} \leq c(p)\, t^{-\frac{n}{2p}}, \quad \text{where} \,\, n = \text{dim}\,\, \gw.
\end{equation}
Since any weak solution of \eqref{pb} as defined is a mild solution \cite{SY} and the global attractor $\mathscr{A}$ is an invariant set, for any $g \in \mathscr{A} \subset E$, we have
\begin{equation}\label{p1c}
	\begin{split}
	&\|S(t) g\|_{L^{\infty}} \leq \|e^{A t}\|_{\mathcal{L} {(L^2, L^\infty)}} \|g\| + \int_{0}^{t} \|e^{A (t - \sigma)} \|_{(L^2, L^\infty)} \|f(S(\sigma) g) - f(S(\sigma) 0)\|\, d\sigma\\
        &\leq c(2) t^{-\frac{3}{4}} \|g\| + \int_{0}^{t} c(2) (t - \sigma)^{-\frac{3}{4}} L(Q_1) (\|S(\sigma) g\|_E + \|S(\sigma) 0\|_E)\, d\sigma, \quad t > 0,
	\end{split}
\end{equation}
where $L(Q_1)$ is the Lipschitz constant of the nonlinear map $f$ restricted on the closed, bounded ball centered at the origin with radius $\sqrt{Q_1}$ in $E$. The global attractor $\mathscr{A}$ is invariant so that
	$$
	\{S(t) \mathscr{A} : t \geq 0\} \subset B_0 \, (\subset H) \cap  B_1 \, (\subset E).
	$$
Then from \eqref{p1c} we obtain
\begin{equation} \label{p1d}
	\begin{split}
	&\|S(t) g\|_{L^{\infty}}  \leq c(2) K_1 t^{-\frac{3}{4}} + \int_{0}^{t} c(2) L(Q_1) \left(\sqrt{Q_1} + \sqrt{Q_2}\right)  (t - \sigma)^{-\frac{3}{4}}\, d\sigma  \\
	& \qquad\; \qquad \; = c(2) [K_1 t^{-\frac{3}{4}} + 4 L(Q_1) \left(\sqrt{Q_1} + \sqrt{Q_2}\right) t^{\frac{1}{4}}], \quad \text{for} \,\, 0 < t \leq 1,
	\end{split}
\end{equation}
where  
$$
	Q_2 = \sup_{0 \leq \sigma \leq t \leq 1} \|S(\sigma) 0\|^2_E.
$$
Take $t = 1$ in \eqref{p1d} and get
$$
	\|S(1) g\|_{L^{\infty}} \leq c(2) \left(K_1 + 4 L(Q_1) \left(\sqrt{Q_1} + \sqrt{Q_2}\right)\right), \quad \text{for any} \; g \in \mathscr{A}.
$$
The invariance of $\mathscr{A}$ implies that $S(1) \ms{A} = \ms{A}$. Therefore, the global attractor $\mathscr{A}$ is a bounded subset in $L^{\infty} (\gw)$.
\end{proof}

\begin{theorem}\label{p2}
	The global attractor $\mathscr{A}$ in the space $H$ for the Hindmarsh-Rose semiflow $\{S(t)\}_{t \geq 0}$ is a bounded set in $H^2(\gw, \mathbb{R}^3)$.
\end{theorem}

\begin{proof} 
Consider the solution trajectories inside the global attractor $\mathscr{A}$.

Step 1. Take the $L^2$ inner-product $\inpt{\eqref{ueq}, u_t}$ to obtain
\begin{equation*}
	\begin{split}
	\|u_t\|^2 +&\, \frac{d_1}{2}\frac{d}{dt}\|\nb u\|^2 = \int_\gw (a u^2 - b u^2 + v - w + J) u_t\, dx\\
	\leq &\, \int_\gw \left(a\, C^2_\infty + b\, C^3_\infty + 2 {C_\infty} + J\right) |u_t|\, dx\\
	= &\, \frac{1}{2} \left(a\, C_\infty^2 + b\, C_\infty^3 + 2 {C_\infty} + J\right)^2 |\gw| + \frac{1}{2} \|u_t\|^2,
	\end{split}	
\end{equation*}
where ${C_\infty}$ is from \eqref{p1a}, for the first component $u(t, x)$ of all the solution trajectories in $\mathscr{A}$. Also take the $L^2$ inner-product $\inpt{\eqref{veq}, v_t}$ to obtain
\begin{equation*}
	\begin{split}
	&\|v_t\|^2 + \frac{d_2}{2} \frac{d}{dt} \|\nb v\|^2 = \int_\gw (\alpha - \beta u^2 - v) v_t\, dx\\
	\leq &\, \int_\gw (\alpha + \beta {C_\infty}^2 + {C_\infty}) |v_t|\, dx = \frac{1}{2} (\alpha + \beta {C_\infty}^2 + {C_\infty})^2 |\gw| + \frac{1}{2}\|v_t\|^2 
	\end{split}
\end{equation*}
for the second component $v(t, x)$ of all the trajectories in $\mathscr{A}$. Then take the $L^2$ inner-product $\inpt{\eqref{weq}, w_t}$ to acquire 
\begin{equation*}
	\begin{split}
	&\|w_t\|^2 + \frac{d_3}{2} \frac{d}{dt} \|\nb w\|^2 = \int_\gw (q u - q c - r w) w_t\, dx\\
	\leq &\, \int_\gw (q {C_\infty} + q |c| + r {C_\infty})|w_t|\, dx = \frac{1}{2} (q {C_\infty} + q |c| + r {C_\infty})^2 |\gw| + \frac{1}{2} \|w_t\|^2
	\end{split}
\end{equation*}
for the third component $w(t, x)$ of all the trajectories in $\mathscr{A}$.
	
Summing up the above three estimates we get
\beq \label{p3}
	\begin{split}
	&\|u_t\|^2 + \|v_t\|^2 + \|w_t\|^2 + \frac{d}{dt} \left\{d_1 \|\nb u\|^2 + d_2 \|\nb v\|^2 + d_3 \|\nb w\|^2\right\} \\[6pt]
	\leq &\, \left((a {C_\infty}^2 + b {C_\infty}^3 + 2 {C_\infty} + J)^2 + (\alpha + \beta {C_\infty}^2 + {C_\infty})^2 + (q {C_\infty} + q |c| + r {C_\infty})^2\right) |\gw|.
	\end{split}
\eeq
Integrating the inequality \eqref{p3} over the time interval $[0,1]$, we obtain
	\beq \label{p4}
	\begin{split}
	&\int_{0}^{1} (\|u_t (s)\|^2 + \|v_t (s)\|^2 + \|w_t (s)\|^2)\, ds\\[3pt]
	\leq &\; d_1 \|\nb u(0)\|^2 + d_2 \|\nb v(0)\|^2 + d_3 \|\nb w(0)\|^2\\[4pt]
	& + (a {C_\infty}^2 + b {C_\infty}^3 + 2 {C_\infty} + J)^2 |\gw| + (\alpha + \beta {C_\infty}^2 + {C_\infty})^2 |\gw| \\[4pt]
	&\, + (q {C_\infty} + q |c| + r {C_\infty})^2 |\gw|\\[3pt]
	\leq &\, (d_1 + d_2 + d_3)Q_1 + (a {C_\infty}^2 + b {C_\infty}^3 + 2 {C_\infty} + J)^2 |\gw| \\[3pt]
	& + (\alpha + \beta {C_\infty}^2 + {C_\infty})^2 |\gw| + (q {C_\infty} + q |c| + r {C_\infty})^2 |\gw|.
	\end{split}
	\eeq
	
Step 2. For the diffusive Hindmarsh-Rose equations confined in the set of the global attractor $\ms{A}$, we can differentiate the equations \eqref{ueq}, \eqref{veq}, and \eqref{weq} to get
\begin{equation} \bl{Bp}
	\begin{split}
	u_{tt} & = d_1 \gd u_t +  2 a u u_t - 3 b u^2 u_t + v_t - w_t,  \\[5pt]
	v_{tt} & = d_2 \gd v_t - 2 \beta u u_t - v_t,  \\[5pt]
	w_{tt} & = d_3 \gd w_t + q u_t - r w_t.
	\end{split}
\end{equation}
Take the inner products $\inpt{\eqref{ueq}, t^2 u_t}, \inpt{\eqref{veq}, t^2 v_t}, \inpt{\eqref{weq}, t^2 w_t}$ and then sum them up,
\begin{equation}\label{p8}
	\begin{split}
	& -t\|u_t\|^2 - t\|v_t\|^2 - t\|w_t\|^2 + \frac{1}{2} \frac{d}{dt} (\|t u_t\|^2 + \|t v_t\|^2 + \|t w_t\|^2) \\[8pt]
	& \qquad + t^2 (d_1 \|\nb u_t\|^2 + d_2 \|\nb v_t\|^2 + d_3 \|\nb w_t\|^2) \\[5pt]
	& = \int_\gw t^2 (2 a u u_t^2 - 3 b u^2 u_t^2 +v_t u_t - w_t u_t - 2\beta u u_t v_t - v_t^2 + q u_t w_t - r w_t^2)\, dx\\
	&\leq\, \int_\gw t^2 \left[2 a {C_\infty} u_t^2 + \frac{1}{2}(v_t^2 + u_t^2) + \frac{1}{2} (w_t^2 + u_t^2) + \beta {C_\infty} (u_t^2 + v_t^2) + \frac{q}{2} (u_t^2 + w_t^2)\right] dx \\
	& = t^2 \left( 2 a {C_\infty} + 1 + \beta {C_\infty} + \frac{q}{2}\right)\|u_t\|^2 + t^2 \left(\frac{1}{2} + \beta {C_\infty}\right)\|v_t\|^2 + t^2 \left(\frac{1}{2} + \frac{q}{2}\right)\|w_t\|^2,
	\end{split}
\end{equation}
where the $u$-component portion is deduced by
\begin{equation}
	\begin{split}
	&\quad -t \|u_t\|^2 + \frac{1}{2} \frac{d}{dt} \|t u_t\|^2  = -t \|u_t\|^2 + \frac{1}{2} \frac{d}{dt} \inpt{t u_t, t u_t}\\[4pt]
	& = -t \|u_t\|^2 + \frac{1}{2} \left(\inpt{\frac{d}{dt} (t u_t), t u_t} + \inpt{t u_t, \frac{d}{dt} (t u_t)}\right) \\
	& = -t \|u_t\|^2 + \inpt{\frac{d}{dt} (t u_t), t u_t} = -t \|u_t\|^2 + \inpt{u_t, t u_t} + \inpt{t u_{tt}, t u_t}\\[6pt]
	& = -t \|u_t\|^2 + t \|u_t\|^2 + \inpt{u_{tt}, t^2 u_t} = \inpt{u_{tt}, t^2 u_t}.
	\end{split}
\end{equation}
Similar derivation goes to the $v$-component and $w$-component portion as well. 
	
Now we integrate the differential inequality \eqref{p8} on $[0, t]$ to obtain
\begin{equation}\label{p9}
	\begin{split}
	&\quad \frac{1}{2} (\|t u_t\|^2 + \|t v_t\|^2 + \|t w_t\|^2)\\
	& \leq \int_{0}^{t} s^2 \left(2 a {C_\infty} + 1 + \beta {C_\infty} +\frac{q}{2}\right)\|u_t(s)\|^2\, ds  \\
	&\quad + \int_{0}^{t} s^2 \left(\frac{1}{2} +\beta {C_\infty}\right)\|v_t(s)\|^2\, ds + \int_{0}^{t} s^2 \left(\frac{1}{2} + \frac{q}{2} \right)\|w_t(s)\|^2\, ds \\
	&\quad + \int_{0}^{t} s (\|u_t(s)\|^2 + \|v_t(s)\|^2 + \|w_t(s)\|^2)\, ds.
	\end{split}
\end{equation}
In the above inequality we can take $t = 1$ and get
\begin{equation} \label{p9}
	\begin{split}
	&\quad \|u_t(1)\|^2 + \|v_t(1)\|^2 + \|w_t(1)\|^2\\[2pt]
	& \leq 2 \int_{0}^{1} \left(2 a {C_\infty} + 1 + \beta {C_\infty} +\frac{q}{2}\right)\|u_t(s)\|^2\, ds \\
	&\quad + 2 \int_{0}^{1} \left(\frac{1}{2} +\beta {C_\infty}\right)\|v_t(s)\|^2\, ds + 2 \int_{0}^{1} \left(\frac{1}{2} + \frac{q}{2} \right)\|w_t(s)\|^2\, ds \\
	&\quad + 2 \int_{0}^{1} (\|u_t(s)\|^2 + \|v_t(s)\|^2 + \|w_t(s)\|^2)\, ds \\
	& \leq 2 \left(2 a {C_\infty} + 5 + 2 \beta {C_\infty} + q \right) \int_{0}^{1} (\|u_t(s)\|^2 + \|v_t(s)\|^2 + \|w_t(s)\|^2)\, ds \leq D
	\end{split}
\end{equation}
where, by the inequality in \eqref{p4} from the Step 1, 
\begin{equation*}
	\begin{split}
	D &\, = 2 \left(2 a {C_\infty} + 5 + 2 \beta {C_\infty} + q \right) \left\{(d_1 + d_2 + d_3)Q_1 \right. \\[5pt]               
	 &\, + (a {C_\infty}^2 + b {C_\infty}^3 + 2 {C_\infty} + J)^2 |\gw|\\[5pt]
	 &\, \left. + (\alpha + \beta {C_\infty}^2 + {C_\infty})^2 |\gw| + (q {C_\infty} + q |c| + r {C_\infty})^2 |\gw| \right\}.
	 \end{split}
\end{equation*}
where $Q_1$ is given in \eqref{ac}.
	
	Step 3.  Since the global attractor $\mathscr{A}$ is an invariant set, for any trajectory $g(t) = (u(t), v(t), w(t)) \in \mathscr{A}$, one has $\tg (t) = g(t-1) \in \mathscr{A}$ such that $g (t)= S(1) \tg (t)$. Then the inequality \eqref{p9} together with the equations \eqref{ueq}, \eqref{veq} and \eqref{weq} implies that
	\begin{equation}\label{p10}
	\begin{split}
	&\quad d_1 \|\gd u(t)\| + d_2 \|\gd v(t)\| + d_3 \|\gd w(t)\|\\[5pt]
	&\leq \|u_t (t)\| + \|v_t (t)\| + \|w_t (t)\| + a \|u^2 (t)\| + b\|u^3 (t) \| + \|v(t)\| + \|w(t)\|  \\[5pt]
	&\quad + \beta \|u^2 (t)| + \|v(t)\| + q \|u(t)\| + r\|w(t)\| + (J + \alpha + q |c|) |\gw|^{\frac{1}{2}}       \\[7pt]
	& = \|\tilde{u}_t(t+1)\| + \|\tilde{v}_t(t+1)\| + \|\tilde{w}_t(t+1)\| + q \|u(t)\| + 2 \|v(t)\|     \\[5pt]
	&\quad + (1 + r)\|w(t)\| + (a + \beta)\|u(t)\|_{L^4}^2 + b \|u(t)\|_{L^6}^3 + (J + \alpha + q |c|) |\gw|^{\frac{1}{2}}  \\[2pt]
	& \leq D^{\frac{1}{2}} + (q + 3 + r) K_1^{\frac{1}{2}} + (a + \beta) K_2^{\frac{1}{2}} + b K_3^{\frac{1}{2}} + (J + \alpha + q |c|)|\gw|^{\frac{1}{2}},
	\end{split}
	\end{equation}
where the positive constants $K_1, K_2, K_3$ are defined in \eqref{abs} of Theorem \ref{Th2} and \eqref{Kp} of Theorem \ref{Th2p}.
	
	Since the Laplacian operator $A_0 = \Delta$ with the Neumann boundary condition \eqref{nbc} is self-adjoint and negative definite modulo constant functions, the Sobolev space norm of any $\varphi \in H^2 (\gw, \mathbb{R}^3)$ is equivalent to $\|\varphi\|^2 + \|\nb {\varphi}\|^2 + \|\gd {\varphi}\|^2$. Therefore, the inequality \eqref{p10} together with Theorem \ref{Th2},  Theorem \ref{Th3}, and Theorem \ref{MTh} shows that the global attractor $\mathscr{A}$ is a bounded set in $H^2 (\gw, \mathbb{R}^3)$.
\end{proof}

\begin{theorem} \label{grd}
	The dynamical system $\{S(t)\}_{t \geq 0}$ generated by the diffusive Hindmarsh-Rose equations \eqref{pb} is a gradient system and the structure of its global attractor $\mathscr{A}$ in $H \cap E$ is given by
\beq \label{str}
	\mathscr{A} = \bigcup_{g \in G} W^u (g)
\eeq	
where $G$ is the set of all the steady states and $W^u (g)$ stands for the unstable manifold associated with the steady state $g$.
\end{theorem}

\begin{proof} 
	It suffices to show that \cite[Definition 10.11]{Rb} there exists a continuous Lyapunov functional $\Gamma$ on a positively invariant set $\mathfrak{S}$ with respect to this semiflow, which contains the global attractor $\ms{A}$ in $H$, such that $\frac{d}{dt} \Gamma (S(t)g) \leq 0$ along any solution trajectory in $\mathfrak{S}$ of the evolutionary equation \eqref{pb} and that if $\Gamma (S(\tau)g) = \Gamma (g)$ for some $\tau > 0$, then $g$ is a steady state.
	
	For this system \eqref{pb}, we can construct the following Lyapunov functional on the global attractor $\ms{A}$:
\beq \label{Lyap}
	\Gamma (g(t)) =  -\left(\frac{1}{2}\, \|\nb g(t) \|^2 + \int_\gw F(g(t, x))\, dx \right)
\eeq
where
$$
	F(g(t, x)) = \int_0^t f (g(s, x) )\cdot dg , \quad \gamma (g) \subset  \ms{A}, 
$$
which is the line integral along the trajectory $\gamma (g) \subset \mathbb{R}^3$ over a time interval $[0, t]$. By the $H^2$-regularity of the global attractor $\ms{A}$ shown in Theorem \ref{p2}, for all solution trajectories $g(t) = (u(t), v(t), w(t))$ of the equation \eqref{pb} in $\ms{A}$, we have
$$
	\frac{d}{dt} \,\Gamma (g(t)) = - \left\langle Ag(t), \frac{dg}{dt} \right\rangle - \left\langle f(g(t)), \frac{dg}{dt} \right\rangle = - \|g_t \|^2 \leq 0, \quad t \geq 0.
$$
If $\Gamma (S(\tau)g) = \Gamma (g)$ for some $\tau > 0$, then $\frac{dg}{dt} = 0$ for almost all $t \in [0, \tau]$, which implies that $g(t) \equiv g(0)$ so that $g$ must be a steady state. 

Moreover, we can prove that the functional $\Gamma: \ms{A} (\subset E) \to \mathbb{R}$ is continuous. Therefore, by Theorem 10.13 in \cite{Rb}, $\Gamma (g)$ is a continuous Lyapunov functional and the Hindmarsh-Rose semiflow a gradient system. Consequently \eqref{str} is proved.
\end{proof}

\begin{remark}
	The global attractor $\ms{A}$ for the Hindmarsh-Rose semiflow has a finite fractal dimension, which has also been proved indirectly by the existence of an exponential attractor for this semiflow in the space $H$. The latter by definition is a set of finite fractal dimension in $H$ and the global attractor must be a subset of any exponential attractor. That proof has been made in a separate paper by the first two authors but not included here in this paper.
\end{remark}

\section{\textbf{Global Attractors for Partly Diffusive Hindmarsh-Rose Equations}}

\vspace{3pt}
In neuronal dynamics, the partly diffusive models \eqref{pHR} or \eqref{qHR} is more commonly interesting, since the ions currents represented by the variables $v(t,x)$ and $w(t,x)$ may or may not diffuse. 

In this section we shall prove the existence of a global attractor for the partly diffusive Hindmarsh-Rose equations \eqref{pHR} and \eqref{qHR}, respectively. Note that the partly diffusive system \eqref{pHR} can be formulated into the evolutionary equation:
\begin{equation} \label{ppb}
 	\begin{split}
   	& \frac{\partial g}{\partial t} = \Hat{A} g + \Hat{f} (g), \quad t > 0, \\[2pt]
    	g &\, (0) = g_0 = (u_0, v_0, w_0) \in H.
	\end{split}
\end{equation}
Here the nonnegative self-adjoint operator
\begin{equation} \label{opAh}
        \Hat{A} =
        \begin{pmatrix}
            d_1 \gd  & 0   & 0 \\[3pt]
            0 & - I   & 0 \\[3pt]
            0 & 0 & - r I 
        \end{pmatrix}
        : D(\Hat{A}) \rightarrow H,
\end{equation}
where $D(\Hat{A}) = \{g \in H^2(\gw) \times L^2 (\gw, \mathbb{R}^2): \pdr g /\pdr \nu = 0 \}$, and
\begin{equation} \label{opfh}
    \Hat{f} (u,v, w) =
        \begin{pmatrix}
             \vp (u) + v - w + J \\[4pt]
            \psi (u),  \\[4pt]
	     q (u - c)
        \end{pmatrix}
        : H^1 (\gw) \times L^2(\gw, \mathbb{R}^2) \longrightarrow H.
\end{equation}

Another partly diffusive system \eqref{qHR} can be formulated into the evolutionary equation:
\begin{equation} \label{ppb}
 	\begin{split}
   	& \frac{\partial g}{\partial t} = \widetilde{A} g + \widetilde{f} (g), \quad t > 0, \\[2pt]
    	g &\, (0) = g_0 = (u_0, v_0, w_0) \in H.
	\end{split}
\end{equation}
Here the nonnegative self-adjoint operator
\begin{equation} \label{opAh}
        \widetilde{A} =
        \begin{pmatrix}
            d_1 \gd  & 0   & 0 \\[3pt]
            0 &  d_2 \gd   & 0 \\[3pt]
            0 & 0 & - r I 
        \end{pmatrix}
        : D(\widetilde{A}) \rightarrow H,
\end{equation}
where $D(\widetilde{A}) = \{g \in H^2(\gw, \mathbb{R}^2) \times L^2 (\gw): \pdr g /\pdr \nu = 0 \}$, and
\begin{equation} \label{opfh}
    \widetilde{f} (u,v, w) =
        \begin{pmatrix}
             \vp (u) + v - w + J \\[4pt]
            \psi (u) - v,  \\[4pt]
	     q (u - c)
        \end{pmatrix}
        : H^1 (\gw, \mathbb{R}^2) \times L^2(\gw) \longrightarrow H.
\end{equation}

Below is the Kolmogorov-Riesz compactness Theorem shown in \cite[Theorem 5]{HO}.

\begin{lemma} \bl{KRTh}
	Let $1 \leq p < \infty$ and $\gw$ be a bounded domain with locally Lipschitz boundary in $\mathbb{R}^n$. A subset $\mathcal{F}$ in $L^p (\gw)$ is precompact if and only if the following two conditions are satisfied\textup{:}

	\textup{1)} $\mathcal{F}$ is a bounded set in $L^p (\gw)$.
	
	\textup{2)} For every $\ve > 0$, there is some $\eta > 0$ such that, for all $f \in \mathcal{F}$ and $y \in \mathbb{R}^n$ with $| y | < \eta$,
$$
	\int_\gw | f(x + y) - f(x)|^p \, dx < \ve^p.
$$
It is a convention that $f(x) = 0$ for $x \in \mathbb{R}^n \backslash \gw$.
\end{lemma}

\begin{theorem} \bl{pHRA}
	There exists a global attractor $\mathscr{A}_1$ in the space $H = L^2 (\gw, \mathbb{R}^3)$ for the semiflow generated by the solutions of the partly diffusive Hindmarsh-Rose equations \eqref{pHR}. 
\end{theorem}

 \begin{proof}
 	Since Corollary \ref{cor1} has proved that there exists an absorbing set for each of the partly diffusive Hindmarsh-Rose system \eqref{pHR} and \eqref{qHR}, it suffices to show that the semiflows generated by these two systems are asymptotically compact via an approach different from Theorem \ref{Th3}, but by means of Theorem \ref{Th2p} and Lemma \ref{KRTh}.
	
	The Laplacian operator $d_1 \gd$ with the Neumann boundary condition generates a parabolic semigroup $e^{d_1 \gd t}, t \geq 0$. The $u$-component of the solutions to \eqref{pHR} and to \eqref{qHR} is expressed by
$$
	u(t) = e^{d_1 \gd t} u_0 + \int_0^t e^{d_1 \gd (t - s)} (\vp (u) + v - w + J)\, ds, \quad t \geq 0.
$$
For $1 \leq p < q$, the $L^p \to L^q$ regularity of parabolic semigroup \cite[Theorem 38.10]{SY} indicates that, for space dimension $n \leq 3$,
\beq \bl{pq}
	\|e^{d_1 \gd t} u_0 \|_{L^q} \leq c(p, q) \, t^{-\frac{3}{2}(\frac{1}{p} - \frac{1}{q})} \|u_0\|_{L^p}, \quad t > 0.
\eeq	

	Step 1. Let $p = 2$ and $q=4$ in \eqref{pq}. We have
$$
	\|e^{d_1 \gd t} u_0 \|_{L^4} \leq \widetilde{c} \, t^{-\frac{3}{8}} \|u_0\|_{L^2}, \quad t > 0.
$$		
where $\widetilde{c}$ is a constant. Then we see
\beq \bl{L24}
	\|u(t)\|_{L^4} = \frac{\widetilde{c}}{t^{3/8}} \|u_0\| + \int_0^t  \| e^{d_1 \gd (T-s)} (\vp (u(s)) + v(s) - w(s) + J)\|_{L^4}\, ds,
\eeq
where $\vp (u) = au^2 - bu^3$. From \eqref{E2}, for the partly diffusive Hndmarsh-Rose equations \eqref{pHR}, we have
\begin{equation*}
	\begin{split}
	&\frac{d}{dt} (C_1 \|u(t)\|^2 + \|v(t)\|^2 + \|w(t)\|^2) + d_1 C_1 \|\nb u(t)\|^2 \\
	+ &\, (C_1 \|u \|^2 + \|v \|^2 + r \|w \|^2) \leq (2C_2 + C_1^2)| \gw|, \;\; t \geq 0.
	\end{split}
\end{equation*}
Set $d_* = \min \{d_1, 1\}$. Then it holds that
\beq \bl{E2p}
	\begin{split}
	\frac{d}{dt} (C_1 \|u(t)\|^2 &\, + \|v(t)\|^2 + \|w(t)\|^2) + d_* C_1 (\|\nb u(t)\|^2 + \|u \|^2)   \\
	+ &\, (\|v \|^2 + r \|w \|^2) \leq (2C_2 + C_1^2)| \gw|, \quad t \geq 0.
	\end{split}
\eeq
Integrate \eqref{E2p} over the time interval $[0, t]$ to get
\beq \bl{nubd}
	 \int_0^t (d_* C_1 \|u\|^2_{H^1(\gw)} + \|v \|^2 + r \|w \|^2)\, ds \leq \max \{C_1, 1\} \|g_0\|^2 + t (2C_2 + C_1^2)| \gw|, \; t \geq 0.
\eeq
Since $H^1(\gw)$ is continuously embedded in $L^4 (\gw)$ and $L^6 (\gw)$, and the $C_0$-semigroup $e^{d_1 \gd t}$ is a contraction semigroup both on $L^2(\gw)$ and $H^1(\gw)$ so that 
$$
	\|e^{d_1 \gd (t-s)}\|_{\mathcal{L}(L^2)} \leq 1 \quad \text{and} \quad \| e^{d_1 \gd (t-s)}\|_{\mathcal{L}(H^1)} \leq 1,
$$	
it follows from \eqref{L24} and \eqref{nubd} that
\beq \bl{Mds}
	\begin{split}
	&\|u(t)\|_{L^4} \leq \frac{\widetilde{c}}{t^{3/8}} \|u_0\| + \int_0^t \|e^{d_1 \gd (t-s)}\|_{\mathcal{L}(L^2)} \|au^2 -bu^3 + v - w + J \|\, ds \\
	\leq &\, \widetilde{C} \left(\frac{1}{t^{3/8}} \|u_0\| + \int_0^t  (\|au^2 -bu^3 + v - w\|^2 + 1)\, ds + J t |\gw|^{1/2} \right)  \\
	= &\, \widetilde{C} \left(\frac{1}{t^{3/8}} \|u_0\| +  \int_0^t (d_* C_1 \|u\|^2_{H^1(\gw)} + \|v \|^2 + r \|w \|^2)\, ds + t + J t |\gw|^{1/2} \right)  \\
	\leq &\, \widetilde{C} \left(\frac{1}{t^{3/8}} \|g_0\| + \max \{C_1, 1\} \|g_0\|^2 + t (2C_2 + C_1^2)| \gw| + t + J t |\gw|^{1/2} \right),  \; t > 0,
	\end{split}
\eeq
where $\widetilde{C} > 0$ is a constant. 
Take $t = 1$ in \eqref{Mds} and we can confirm that for any given bounded set $B \subset H$ and any initial state $g_0 \in B$, the $u$-component function at time $t=1$ of all these solutions $g(t, g_0)$ are uniformly bounded in the space $L^4 (\gw)$,
\beq \bl{ut1}
	\|u(1)\|_{L^4} \leq \widetilde{C} \left(\interleave B \interleave + \max \{C_1, 1\}\interleave B \interleave^2 + (2C_2 + C_1^2)| \gw| + 1 + J|\gw|^{1/2} \right)
\eeq 
where $\interleave B \interleave = \sup_{g_0 \in B} \|g_0\|$.

The uniform boundedness \eqref{ut1} allows us to use and adapt the trajectory-wise estimates in Theorem \ref{Th2p}. From \eqref{u2}, \eqref{2p3} and the fact that $\frac{1}{3}(2p + 2) = 2$ for $p=2$, we can improve the inequality \eqref{u3} for this system \eqref{pHR} and get

\beq \bl{pu} 
	\begin{split}
	&\frac{1}{4} \, \frac{d}{dt} \|u(t) \|^4_{L^4} + 3d_1 \|u \nb u \|^2 = \int_\gw [a u^5 - b u^6 + u^3 (v - w + J)]\, dx \\
	\leq &\, C_b (a^6 + C_r + J^2) |\gw| - \frac{3}{4} \int_\gw b u^6\, dx + \frac{1}{4} \int_\gw (v^2(t, x) + r w^2(t, x))\, dx \\
	\leq &\, C_b (a^6 + C_r + J^2) |\gw| + \|v(t)\|^2 + r \|w(t)\|^2  - \frac{3}{4} \int_\gw b u^6\, dx   \\
	\leq  &\, C_b (a^6 + C_r + J^2) |\gw| + K_1 - \frac{3}{4} \int_\gw b u^6\, dx,  \quad t \geq T_B + 1,
	\end{split}
\eeq
because Corollary \ref{cor1} shows that for any bounded set $B \subset H$ there is a time $T_B > 0$ such that for $t \geq T_B$, $\|v(t)\|^2 + \|w(t)\|^2$ is uniformly bounded by $K_1$ given in \eqref{lsp2}. 

Note that $ \frac{1}{4} \int_\gw bu^6 dx \geq \frac{1}{4} \int_\gw u^4 dx - \frac{1}{4b} |\gw |$. Then from \eqref{pu} we obtain
\beq \bl{puc}
	 \frac{d}{dt} \|u(t) \|^4_{L^4} + \|u(t) \|^4_{L^4} + 2b \int_\gw u^6\, dx \leq K, \quad t \geq T_B + 1,
\eeq
where 
$$
	K = \left( 4C_b (a^6 + C_r + J^2) + \frac{1}{b} \right) |\gw| + 4 K_1.
$$	
Apply Gronwall inequality to \eqref{puc} without $2b \int_\gw u^6\, dx$ and use \eqref{ut1} to get 
\beq \bl{u4bd}
	\begin{split}
	&\|u(t) \|^4_{L^4} \leq e^{- (t-1)} \|u(1)\|^4_{L^4} + K   \\[4pt]
	\leq &\, \widetilde{C}^4 (\interleave B \interleave + \max \{C_1, 1\} \interleave B \interleave^2 + (2C_2 + C_1^2)| \gw| + 1 + J|\gw|^{1/2})^4 + K
	\end{split}
\eeq
for $t \geq T_B + 1$ and any $g_0 \in B$. Moreover, integrating \eqref{puc} yields the following important bound to be used a little later: For $t \geq T_B + 1$,
\beq \bl{u6bd}
	\begin{split}
	& \int_{T_B +1}^t e^{-(t-s)} \int_\gw u^6 (s, x)\, dx\, ds  \leq \frac{1}{2b} \left( \|u(T_B +1)\|^4_{L^4} + K \right)      \\
	\leq \frac{1}{2b} & \left(\widetilde{C}^4 (\interleave B \interleave + \max \{C_1, 1\} \interleave B \interleave^2 + (2C_2 + C_1^2)| \gw| + 1 + J|\gw|^{1/2})^4 + 2K\right). 
	 \end{split}
\eeq
The inequality \eqref{u4bd} shows that, for any given bounded set $B \subset H$, 
$$
	\bigcup_{t \geq T_B + 1} \left(\bigcup_{g_0 \in B} u(t, \cdot)\right) \;\; \text{is a bounded set in} \; L^4 (\gw)
$$
so that, by the compact embedding $L^4 (\gw) \hookrightarrow L^2 (\gw)$ and $L^4 (\gw) \hookrightarrow L^3 (\gw)$,
\beq \bl{ucmp}
	\bigcup_{t\, \geq T_B + 1} \left(\bigcup_{g_0 \in B} u(t, \cdot)\right) \;\; \text{is a precompact set in} \; L^2 (\gw) \; \text{and in} \; L^3 (\gw).
\eeq

	Step 2. It remains to prove the precompactness of the two component functions 
$v(t, x)$ and $w(t, x)$, which satisfy the ordinary differential equations in \eqref{pHR}. By the variation-of-constant formula for the solutions of ODE, we have
\begin{equation} \bl{vcw}
	\begin{split}
	v(t, x) &= e^{-t} v_0 (x) + \int_0^t e^{-(t-s)} (\alpha - \beta u^2)\, ds \leq \alpha + e^{-t} v_0 - \beta \int_0^t e^{-(t-s)} u^2(s, x)\, ds, \\
	w(t, x) &= e^{- rt} w_0 (x) + \int_0^t e^{-r (t-s)} q (u - c)\, ds \leq \frac{q |c|}{r} + e^{- rt} w_0 + q \int_0^t e^{-r (t-s)} u(s, x)\, ds.
	\end{split}
\end{equation}
By Lemma \ref{KRTh} and \eqref{ucmp}, for any $\ve > 0$, there is some $\eta > 0$ such that, for any given bounded set $B \subset H$ and all $g_0 \in B$, and for $y \in \mathbb{R}^3$ with $| y | < \eta$, it holds that 
\beq \bl{utx}
	\int_\gw | u(t, x + y) - u(t, x)|^3 \, dx < \ve^3, \quad \text{for all} \;\, t \geq T_B + 1.
\eeq
Consider \eqref{vcw} on the time interval $[T_B + 1, \infty)$. Using the H\"{o}lder inequality we can infer that, for any $t \geq T_B +1$ and any $g_0 \in B$,
\begin{equation} \bl{kyi}
	\begin{split}
	&\int_\gw |v(t, x+y) - v(t, x)|^2 dx = e^{-(t - T_B -1)} \int_\gw |v(T_B + 1, x + y) - v(T_B + 1, x)|^2 \,dx \\
	+ &\, \beta \int_{T_B + 1}^t e^{-(t-s)} \int_\gw |u^2(s, x+ y) - u^2 (s, x)|^2 \, dx\, ds  \leq 2\, e^{-(t-T_B -1)} \|v (T_B +1)\|^2    \\
        + &\, \beta \int_{T_B + 1}^t e^{-(t-s)} \int_\gw |u(s, x+ y) - u (s, x)|^2 |u(s, x+ y) + u (s, x)|^2 \,dx\, ds    \\[2pt]
        \leq &\, 2\, e^{-(t-T_B -1)} \|v (T_B +1)\|^2  \\
        + &\, \beta \int_{T_B + 1}^t e^{-(t-s)} \|(u(s, x+y) - u(s, x))^2\|_{L^{3/2}} \|(u(s, x+y) + u(s,x))^2\|_{L^3} \,ds   \\[2pt]
        \leq &\, 2\, e^{-(t-T_B -1)} \|v (T_B +1)\|^2  \\
        + &\, 24 \beta \int_{T_B + 1}^t e^{-(t-s)} \|u(s, x+y) - u(s, x)\|^{2}_{L^3} \left(\|u(s, x+y)\|_{L^6}^{2} + \|u(s, x)\|_{L^6}^{2}\right)\, ds \\
       \leq &\, 2\, e^{-(t-T_B -1)} \|v (T_B +1)\|^2  \\
       + &\, 48 \beta \int_{T_B + 1}^t e^{-(t-s)} \|u(s, x+y) - u(t, x)\|^{2}_{L^3} \|u(s, x)\|_{L^6}^{2}\, ds \\
        \leq &\, 2\, e^{-(t-T_B -1)} \|v (T_B +1)\|^2  \\
        + &\, 48 \beta \int_{T_B + 1}^t e^{-(t-s)} \|u(s, x+y) - u(t, x)\|^{2}_{L^3} (\|u(s, x)\|_{L^6}^{6} + 1)\, ds
        	\end{split}
\end{equation} 
where in the last step of the above inequality, we used Young's inequality
$$
	\|u(s, \cdot )\|_{L^6}^{2} = \left(\int_\gw u^6 (s, x)\,dx \right)^{1/3} \leq \frac{1}{3} \|u(s, \cdot )\|^6_{L^6} + \frac{2}{3} \leq \|u(s, \cdot )\|_{L^6}^{6} + 1.
$$
By \eqref{u6bd}, \eqref{utx} and \eqref{kyi}, for any $\ve > 0$, there is some $\eta > 0$ such that, for any given bounded set $B \subset H$ and all $g_0 \in B$, and for $y \in \mathbb{R}^3$ with $| y | < \eta$, we have
\beq \bl{vtx}
	\begin{split}
	&\int_\gw |v(t, x+y) - v(t, x)|^2 dx \\
	\leq &\, 2\, e^{-(t-T_B -1)} \|v (T_B +1)\|^2 + 48\, \beta \int_{T_B + 1}^t e^{-(t-s)} \,\ve^2 \, \left(\int_\gw u(s, s)^6\,dx + 1\right)\, ds \\
	= &\, 2\, e^{-(t-T_B -1)} \|v (T_B +1)\|^2 + 48\, \beta \, \ve^2 \left(K^*+ 1\right), \quad t \geq T_B +1, \; g_0 \in B,
	\end{split}
\eeq
where the constant $K^*$ is given by the right-hand side of \eqref{u6bd},
$$
	K^* = \frac{1}{2b} \left(\widetilde{C}^4 (\interleave B \interleave + \max \{C_1, 1\} \interleave B \interleave^2 + (2C_2 + C_1^2)| \gw| + 1 + J|\gw|^{1/2})^4 + 2K\right). 
$$

	Moreover, there exists a time 
$$
	T^*(B) = T_B + 1 + \log_e \left(\frac{\ve^2}{4 K_1 } \right)
$$ 
such that 
\beq \bl{vTB}
	2\, e^{-(t-T_B -1)} \|v (T_B +1)\|^2 < \ve^2, \quad \text{for} \; t \geq T^*(B),
\eeq
where $K_1$ is given in \eqref{lsp2}. It follows from \eqref{vtx} and \eqref{vTB} that 
\beq \bl{vcp}
	\int_\gw |v(t, x+y) - v(t, x)|^2 dx < \left[ 1 + 48\, \beta (K^* + 1) \right] \ve^2, \quad t \geq T^*(B), \, g_0 \in B.
\eeq
Since $\ve > 0$ is arbitrary, by Lemma \ref{KRTh}, \eqref{vcp} demonstrates that 
\beq \bl{vcmp}
	\bigcup_{t\, \geq \, T^*(B)} \left(\bigcup_{g_0 \in B} v(t, \cdot)\right) \;\, \text{is precompact in}\; L^2 (\gw).
\eeq

Similarly, by Lemma \ref{KRTh} and \eqref{ucmp}, for any $\ve > 0$, there is some $\eta > 0$ such that, for any given bounded set $B \subset H$ and all $g_0 \in B$, and for $y \in \mathbb{R}^3$ with $| y | < \eta$, it holds that 
$$
	\int_\gw | u(t, x + y) - u(t, x)|^2 \, dx < \ve^2, \quad \text{for all} \;\, t \geq T_B + 1,
$$
and we can show that, for any $g_0 \in B$,
\begin{equation} \bl{wtx}
	\begin{split}
	&\int_\gw |w(t, x+y) - w(t, x)|^2 dx \leq 2\, e^{- r(t-T_B -1)} \|w (T_B +1)\|^2   \\
	+ &\, q \int_{T_B + 1}^t e^{- r(t-s)} \int_\gw |u(s, x+ y) - u(s, x)|^2 dx\, ds < \left(1 + \frac{q}{r}\right) \ve^2, \; \,\; t \geq \widetilde{T}(B),
	\end{split}
\end{equation}
where 
$$
	\widetilde{T}(B) = T_B + 1 + \frac{1}{r} \log_e \left(\frac{\ve^2}{4 K_1 } \right).
$$
\eqref{wtx} shows that 
\beq \bl{wcmp}
	\bigcup_{t\, \geq \, \widetilde{T}(B)} \left(\bigcup_{g_0 \in B} w(t, \cdot)\right) \;\, \text{is precompact in}\; L^2 (\gw).
\eeq
Finally, put together \eqref{ucmp}, \eqref{vcmp} and \eqref{wcmp}. Then we see
\beq \bl{gcmp}
	\bigcup_{t\, \geq \, \max \{T^*(B), \, \widetilde{T}(B)\}} \left(\bigcup_{g_0 \in B} g(t, \cdot)\right) \;\, \text{is precompact in}\; H.
\eeq
Therefore, the solution semiflow generated by the system \eqref{pHR} is asymptotically compact. By Proposition \ref{L:basic}, there exists a global attractor $\mathscr{A}_1$ in the space $H = L^2 (\gw, \mathbb{R}^3)$ for the partly diffusive Hindmarsh-Rose equations \eqref{pHR}. 
\end{proof}

\begin{theorem} \bl{qHRA}
	There exists a global attractor $\mathscr{A}_2$ in the space $H = L^2 (\gw, \mathbb{R}^3)$ for the semiflow generated by the solutions of the partly diffusive Hindmarsh-Rose equations \eqref{qHR}. 
\end{theorem}

\begin{proof}
	The semiflow generated by the solutions of the system \eqref{qHR} is also dissipative since Corollary \ref{cor1} with $d_2 \|\nb v(t)\|^2$ added to the right-hand side of \eqref{pE2} shows that there exists an absorbing set in $H$ for this system. The proof of the asymptotic compactness of the $u$-component functions and the $w$-component functions for this system \eqref{qHR} is the same as in the proof of Theorem \ref{pHRA}. 
	
	Thus it suffices to show the asymptotic compactness of the $v$-component functions for this system. Since Theorem \ref{Th2p} and \eqref{vK2} already proved the $v$-absorbing property for each solution trajectory,
\beq \bl{vtp}
	\limsup_{t \to \infty} \|v(t)\|^4_{L^4} < K_2, \quad \text{for any given} \,\; g_0 \in H, 
\eeq	
we need only to show that for any given bounded set $B \subset H$ and all the initial states $g_0 \in B$, the bunch of $v$-component functions of all these solutions $g(t, g_0)$ admits a uniform bound in the space $L^4 (\gw)$ at the unified time point $t = 1$. This will be similar to \eqref{ut1} in the proof of Theorem \ref{pHRA}.

According to \eqref{pq}, here we have
$$
	\|e^{d_2 \gd t} v_0 \|_{L^4} \leq \widetilde{c} \, t^{-\frac{3}{8}} \|v_0\|_{L^2}, \quad t > 0,
$$		
where $\widetilde{c}$ is a constant, and
\beq \bl{L26}
	\|v(t)\|_{L^4} = \frac{\widetilde{c}}{t^{3/8}} \|v_0\| + \int_0^t  \| e^{d_2 \gd (T-s)} (\alpha - \beta u^2(s))\|_{L^4}\, ds, \quad  t > 0.
\eeq
Set $d_* = \min \{d_1, d_2, 1\}$. Adapt \eqref{pE2} to the following inequality
\beq \bl{E3p}
	\begin{split}
	&\frac{d}{dt} (C_1 \|u(t)\|^2 + \|v(t)\|^2 + \|w(t)\|^2) + d_* C_1 ( \|\nb u \|^2 + \|u \|^2)   \\[2pt]
	&+ d_* ( \|\nb v \|^2  + \|v \|^2) + r \|w \|^2 \leq (2C_2 + C_1^2)| \gw|, \quad t \geq 0.
	\end{split}
\eeq
Integrate \eqref{E3p} over the time interval $[0, t]$ to yield
\beq \bl{nvbd}
	 \int_0^t (d_* (C_1 \|u\|^2_{H^1(\gw)} + \|v \|^2_{H^1(\gw)} )+ r \|w \|^2)\, ds \leq \max \{C_1, 1\} \interleave B \interleave^2 + t\, (2C_2 + C_1^2)| \gw|
\eeq
By the fact that $L^4 (\gw), L^6 (\gw)$ are continuously embedded in $H^1(\gw)$, it follows from \eqref{L26} and \eqref{nvbd} that, for $t > 0$,
\beq \bl{MD}
	\begin{split}
	&\|v(t)\|_{L^4} \leq \frac{\widetilde{c}}{t^{3/8}}\, \|v_0\|  + \alpha t\, |\gw|^{1/2} + \beta \int_0^t \|e^{d_2 \gd (t-s)} u^2(s) \|_{L^4}\, ds \\
	\leq &\, \frac{\widetilde{c}}{t^{3/8}} \|v_0\|  + \alpha t\, |\gw|^{1/2} + \beta \int_0^t \|e^{d_2 \gd (t-s)}\|_{\mathcal{L}(L^2, L^4)}\, \| u^2(s) \|_{L^2}\, ds \\
	= &\, \frac{\widetilde{c}}{t^{3/8}} \|v_0\|  + \alpha t\, |\gw|^{1/2} + \beta \int_0^t \|e^{d_2 \gd (t-s)}\|_{\mathcal{L}(L^2, L^4)}\, \| u(s) \|^2_{L^4}\, ds \\
	\leq &\, \frac{\widetilde{c}}{t^{3/8}} \|v_0\|  + \alpha t\, |\gw|^{1/2} + \beta \kappa \int_0^t \|e^{d_2 \gd (t-s)}\|_{\mathcal{L}(H^1)}\, \| u(s) \|^2_{H^1}\, ds   \\
	\leq &\, \frac{\widetilde{c}}{t^{3/8}} \interleave B \interleave  + \alpha t\, |\gw|^{1/2} + \frac{\beta \kappa}{d_* C_1} \left(\max \{C_1, 1\} \interleave B \interleave^2 +\, t (2C_2 + C_1^2)| \gw| \right),
	\end{split} 
\eeq
where $\kappa > 0$ is the $H^1 \hookrightarrow L^4$ embedding constant and $e^{d_2 \gd t}$ is a contraction semigroup on $H^1 (\gw)$. Take $t = 1$ in \eqref{MD} and we reach a uniform bound
$$
	\sup_{g_0 \in B} \|v(1)\|_{L^4} \leq \left(\widetilde{c} +  \frac{\beta \kappa}{d_* C_1} \left(\max \{C_1, 1\}\right) \right) \interleave B \interleave + \alpha |\gw|^{1/2} + \frac{\beta \kappa}{d_* C_1} (2C_2 + C_1^2)| \gw|. 
$$ 
for any given bounded set $B \subset H$ and all $g_0 = (u_0, v_0, w_0) \in B$. Now the bunch of $v$-component functions at time $t =1$ of all these solutions are uniformly bounded in the space $L^4 (\gw)$. Then the argument in Theorem \ref{Th2p} shows that there exists an bounded absorbing set in $L^4 (\gw)$ for the $v$-component functions of all these solutions initiated from the given bounded set $B$ in $H$, which in turn shows that the set of all these $v$-component functions are asymptotically compact in the space $L^2 (\gw)$ by the compact embedding. Thus the proof is completed.	
\end{proof}

\begin{remark}
	Remark 1 at the end of Section 4 said that the global attractor $\mathscr{A}$ for the Hindmarsh-Rose semiflow associated with the diffusive system \eqref{pb} has a finite fractal dimension in $H$ via proof of the existence of an exponential attractor. It is a conjecture that the global attractors $\ms{A}_1$ and $\ms{A}_2$ for the partly diffusive Hindmarsh-Rose equations also have finite fractal dimensions in the space $H$, but that is not pursued in this paper.
\end{remark}

\textbf{Acknowledgment.}  Jianzhong Su is partially supported by USDA Grant number 2018-38422-28564.

\vspace{10pt}
\bibliographystyle{amsplain}

\end{document}